\documentclass[a4paper,oneside,11pt]{article}

\usepackage[a4paper]{geometry}
\usepackage{aeguill}                 
\usepackage{graphicx}
\usepackage{amsmath,amsfonts,amssymb,amsthm}
\usepackage{pstricks}
\usepackage{epstopdf}
\usepackage[utf8]{inputenc} 
\usepackage[T1]{fontenc} 
 
\usepackage[english]{babel} 

\title{The horoboundary of outer space, and growth under random automorphisms}
\author{Camille Horbez}

\usepackage[all]{xy}
\usepackage{bbm}
\begin{document}
\maketitle
\newtheorem{de}{Definition} [section]
\newtheorem{theo}[de]{Theorem} 
\newtheorem{prop}[de]{Proposition}
\newtheorem{lemma}[de]{Lemma}
\newtheorem{cor}[de]{Corollary}
\newtheorem{propd}[de]{Proposition-Definition}

\theoremstyle{remark}
\newtheorem{rk}[de]{Remark}
\newtheorem{ex}[de]{Example}
\newtheorem{question}[de]{Question}

\normalsize

\addtolength\topmargin{-.5in}
\addtolength\textheight{1.in}
\addtolength\oddsidemargin{-.045\textwidth}
\addtolength\textwidth{.09\textwidth}

\begin{abstract}
We show that the horoboundary of outer space for the Lipschitz metric is a quotient of Culler and Morgan's classical boundary, two trees being identified whenever their translation length functions are homothetic in restriction to the set of primitive elements of $F_N$. We identify the set of Busemann points with the set of trees with dense orbits. We also investigate a few properties of the horoboundary of outer space for the backward Lipschitz metric, and show in particular that it is infinite-dimensional when $N\ge 3$. We then use our description of the horoboundary of outer space to derive an analogue of a theorem of Furstenberg and Kifer \cite{FK83} and Hennion \cite{Hen84} for random products of outer automorphisms of $F_N$, that estimates possible growth rates of conjugacy classes of elements of $F_N$ under such products.
\end{abstract}

\setcounter{tocdepth}{2}
\tableofcontents

\section*{Introduction}

Over the past decades, the study of the group $\text{Out}(F_N)$ of outer automorphisms of a free group of rank $N$ has benefited a lot from the study of its action on some geometric complexes, among which stands Culler and Vogtmann's outer space \cite{CV86}. A main source of inspiration in this study comes from analogies with arithmetic groups acting on symmetric spaces, and mapping class groups of surfaces acting on Teichmüller spaces. \emph{Outer space} $CV_N$ (or its unprojectivized version $cv_N$) is the space of equivariant homothety (isometry) classes of simplicial free, minimal, isometric actions of $F_N$ on simplicial metric trees. It is naturally equipped with an asymmetric metric $d$ (i.e. $d$ satisfies the separation axiom and the triangle inequality, but we can have $d(x,y)\neq d(y,x)$). This metric is defined in analogy with Thurston's asymmetric metric on $Teich(S)$. The distance between two trees $T,T'\in CV_N$ is the logarithm of the infimal Lipschitz constant of an $F_N$-equivariant map from the covolume one representative of $T$ to the covolume one representative of $T'$ \cite{FM11}. We aim at giving a description of the horoboundary of outer space, which we then use to derive a statement about the growth of elements of $F_N$ under random products of automorphisms, analogous to a theorem of Furstenberg and Kifer \cite{FK83} and Hennion \cite{Hen84} about random products of matrices.

The horoboundary of a metric space was introduced by Gromov in \cite{Gro81}. Let $(X,d)$ be a metric space, and $b$ be a basepoint in $X$. Associated to any $z\in X$ is a continuous map 
\begin{displaymath}
\begin{array}{cccc}
\psi_z :&X&\to &\mathbb{R}\\
&x&\mapsto &d(x,z)-d(b,z)
\end{array}.
\end{displaymath}

\noindent Let $\mathcal{C}(X)$ be the space of real-valued continuous functions on $X$, equipped with the topology of uniform convergence on compact sets. Under some geometric assumptions on $X$, the map
\begin{displaymath}
\begin{array}{cccc}
\psi:&X&\to&\mathcal{C}(X)\\
&z&\mapsto &\psi_z
\end{array}
\end{displaymath}

\noindent  is an embedding, and taking the closure of its image yields a compactification of $X$, called the \emph{horofunction compactification}. The space $\overline{\psi(X)}\smallsetminus\psi(X)$ is called the \emph{horoboundary} of $X$. In \cite{Wal11}, Walsh extended this notion to the case of asymmetric metric spaces. 

Walsh identified the horofunction compactification of the Teichmüller space of a compact surface, with respect to Thurston's asymmetric metric, with Thurston's compactification, defined as follows (see \cite{FLP79}). Let $\mathcal{C}(S)$ denote the set of free homotopy classes of simple closed curves on $S$. The space $\text{Teich}(S)$ embeds into $\mathbb{PR}^{\mathcal{C}(S)}$ by sending any element to the collection of all lengths of geodesic representatives of homotopy classes of simple closed curves, and the image of this embedding has compact closure. Thurston has identified the boundary with the space of projectivized measured laminations on $S$. 

In the context of group actions on trees, lengths of curves are replaced by translation lengths of elements of the group. The \emph{translation length} of an element $g$ of a group $G$ acting by isometries on an $\mathbb{R}$-tree $T$ is defined as $||g||_T:=\inf_{x\in T}d_T(x,gx)$. Looking at the translation lengths of all elements of $F_N$ yields an embedding of $cv_N$ into $\mathbb{R}^{F_N}$, whose image has projectively compact closure, as was proved by Culler and Morgan \cite{CM87}. This compactification $\overline{CV_N}$ of outer space was described by Cohen and Lustig \cite{CL95} and Bestvina and Feighn \cite{BF94} as the space of homothety classes of minimal, very small, isometric actions of $F_N$ on $\mathbb{R}$-trees, see also \cite{Hor14-5}.

We prove that Culler and Morgan's compactification of outer space is not isomorphic to the horofunction compactification. To get the horocompactification of outer space, one has to restrict translation length functions to the set $\mathcal{P}_N$ of primitive elements of $F_N$, i.e. those elements that belong to some free basis of $F_N$. This yields an embedding of $CV_N$ into $\mathbb{PR}^{\mathcal{P}_N}$, whose image has compact closure $\overline{CV_N}^{prim}$, called the \emph{primitive compactification} \cite{Hor14-1}. Alternatively, the space $\overline{CV_N}^{prim}$ is the quotient of $\overline{CV_N}$ obtained by identifying two trees whenever their translation length functions are equal in restriction to $\mathcal{P}_N$. An explicit description of this equivalence relation in terms of trees was given in \cite[Theorem 0.2]{Hor14-1}. The equivalence class of a tree with dense $F_N$-orbits consists of a single point. The typical example of a nontrivial equivalence class is obtained by equivariantly folding an edge $e$ of the Bass--Serre tree of a splitting of the form $F_N=F_{N-1}\ast$ along some translate $ge$, where $g\in F_{N-1}$ is not contained in any proper free factor of $F_{N-1}$.

\begin{theo} \label{intro-horocompactification}
There exists a unique $\text{Out}(F_N)$-equivariant homeomorphism from $\overline{CV_N}^{prim}$ to the horocompactification of $CV_N$ which restricts to the identity on $CV_N$. For all $z\in\overline{CV_N}^{prim}$, the horofunction associated to $z$ is given by $$\psi_z(x)= \log \sup_{g\in \mathcal{P}_N}\frac{||g||_z}{||g||_x}-\log\sup_{g\in \mathcal{P}_N}\frac{||g||_z}{||g||_b}$$ for all $x\in CV_N$ (identified with its covolume $1$ representative).
\end{theo}

Both suprema in the above formula can be taken over a finite set of elements that only depends on $x$ and $b$. We could also choose any representative of $z$ in $\overline{cv_N}$, and take the supremum over all elements of $F_N$. Denoting by $\text{Lip}(x,z)$ the infimal Lipschitz constant of an $F_N$-equivariant map from $x$ to a fixed representative of $z$ in $\overline{cv_N}$, we also have $$\psi_z(x)=\log\text{Lip}(x,z)-\log\text{Lip}(b,z).$$

A special class of horofunctions in the horoboundary of a metric space $X$ comes from points arising as limits of infinite almost-geodesic rays in $X$, called \emph{Busemann points} \cite{Rie02}. Walsh proved that all points in the horoboundary of the Teichmüller space of a compact surface are Busemann. This is no longer true in outer space, one obstruction coming from the noncompleteness of outer space, see \cite{AK12}: some points in the boundary are reached in finite time along geodesic intervals. We show that Busemann points in the horoboundary of outer space coincide with trees having dense orbits under the $F_N$-action. 

As the Lipschitz metric on outer space is not symmetric, one can also consider the horoboundary of outer space for the backward metric. We investigate some of its properties, but we only give a complete description when $N=2$. There seems to be some kind of duality between the two boundaries we get, the horofunctions for the backward metric being expressed in terms of dual currents. Topologically though, both boundaries are of rather different nature. For example, we show that the backward horocompactification has infinite topological dimension when $N\ge 3$, while the forward horocompactification of outer space has dimension $3N-4$.
\\
\\
\indent Our motivation for understanding the horoboundary of outer space comes from the question of describing the behaviour of random walks on $\text{Out}(F_N)$. Karlsson and Ledrappier proved that a typical trajectory of the random walk on a locally compact group $G$ acting by isometries on a proper metric space $X$ follows a (random) direction, given by a point in the horoboundary of $X$, see \cite{KL06,KL11-2}. 

Given a probability measure $\mu$ on a group $G$, the \emph{left random walk} on $(G,\mu)$ is the Markov chain on $G$ whose initial distribution is given by the Dirac measure at the identity element of $G$, with transition probabilities given by $p(x,y):=\mu(yx^{-1})$. In other words, the position $\Phi_n$ of the random walk at time $n$ is given from its position $e$ at time $0$ by multiplying successively $n$ independent increments $\phi_i$ of law $\mu$ on the left, i.e. $\Phi_n=\phi_n\dots\phi_1$. 

Random walks on linear groups were first considered by Furstenberg and Kesten \cite{FK60}, who studied the asymptotic behaviour of the norm $||X_n\dots X_1||$ of a product of independent matrix-valued random variables. Furstenberg then studied the growth of the vector norms $||X_n\dots X_1v||$, where $v\in\mathbb{R}^N$, along typical sample paths of the random walk on $(SL(N,\mathbb{R}),\mu)$, where $\mu$ is a probability measure. He showed \cite[Theorems 8.5 and 8.6]{Fur63-2} that if $\mu$ satisfies some moment condition, and if the support of $\mu$ generates a noncompact \emph{irreducible} subgroup of $SL(N,\mathbb{R})$, then all vectors in $\mathbb{R}^N\smallsetminus\{0\}$ have the same positive exponential growth rate along typical sample paths of the random walk. Here, a subgroup is \emph{irreducible} if it does not virtually preserve any proper linear subspace of $\mathbb{R}^N$. An analogue of Furstenberg's result for random products in the mapping class group of a compact surface $S$, was provided by Karlsson \cite[Corollary 4]{Kar12}. More precisely, Karlsson showed (again under some moment and irreducibility condition) that the lengths of all isotopy classes of essential simple closed curves on $S$ have the same exponential growth rate. Karlsson derived this statement from Karlsson and Ledrappier's ergodic theorem, by using Walsh's description of the horoboundary of the Teichmüller space of $S$, equipped with Thurston's asymmetric metric.

We use our description of the horoboundary of outer space to study the growth of conjugacy classes of elements of $F_N$ under random products of outer automorphisms, and prove an analogue of Furstenberg's and Karlsson's results in this context. Let $\mu$ be a probability measure on $\text{Out}(F_N)$. In the (generic) case where the support of $\mu$ generates a \emph{nonelementary} subgroup of $\text{Out}(F_N)$, we show that all elements of $F_N$ have the same positive exponential growth rate along typical sample paths of the random walk on $(\text{Out}(F_N),\mu)$. Here, by \emph{nonelementary}, we mean a subgroup of $\text{Out}(F_N)$ which does not virtually fix the conjugacy class of any finitely generated subgroup of $F_N$ of infinite index (this is the analogue of Furstenberg's irreducibility condition). The length $||g||$ of an element $g\in F_N$ (or more precisely of its conjugacy class) is defined as the word length of the cyclically reduced representative of $g$ in the standard basis of $F_N$. The group $\text{Out}(F_N)$ acts on the set of conjugacy classes of elements of $F_N$. In the following statement, we denote by $d_{CV_N}^{sym}$ the symmetrized version of the Lipschitz metric on $CV_N$, defined by setting $d_{CV_N}^{sym}(T,T'):=d_{CV_N}(T,T')+d_{CV_N}(T',T)$.

\begin{theo} \label{intro-random}
Let $\mu$ be a probability measure on $\text{Out}(F_N)$, whose support is finite and generates a non virtually cyclic, nonelementary subgroup of $\text{Out}(F_N)$. Then there exists a (deterministic) constant $\lambda>0$ such that for all $g\in F_N$, and almost every sample path $(\Phi_n)_{n\in\mathbb{N}}$ of the random walk on $(\text{Out}(F_N),\mu)$, we have $$\lim_{n\to +\infty}\frac{1}{n}\log{||\Phi_n(g)||}=\lambda.$$
\end{theo}

The growth rate $\lambda$ is equal to the drift of the random walk on $CV_N$ for the asymmetric Lipschitz metric.
\\
\\
\indent Even if we no longer assume the subgroup generated by the support of $\mu$ to be nonelementary, we can still provide information about possible growth rates of elements of $F_N$ under random automorphisms. This time, several growth rates can arise, and we give a bound on their number. Our main result is an analogue of a version of Oseledets' multiplicative theorem, that is due to Furstenberg and Kifer \cite[Theorem 3.9]{FK83} and Hennion \cite[Théorème 1]{Hen84} in the case of random products of matrices. Given a probability measure $\mu$ on the linear group $GL(N,\mathbb{R})$, Furstenberg--Kifer and Hennion's theorem states that there exists a (deterministic) filtration of $\mathbb{R}^N$ by linear subspaces $\{0\}=L_0\subseteq L_1\subseteq\dots\subseteq L_r=\mathbb{R}^N$, and (deterministic) \emph{Lyapunov exponents} $0\le\lambda_1<\dots<\lambda_r$, so that for all $i\in\{1,\dots,r\}$, all $v\in L_i\smallsetminus L_{i-1}$, and almost every sample path $(A_n)_{n\in\mathbb{N}}$ of the left random walk on $(GL(N,\mathbb{R}),\mu)$, we have $$\lim_{n\to +\infty}\frac{1}{n}\log||A_n v||=\lambda_i.$$ In the case of free groups, the filtration of $\mathbb{R}^N$ is replaced by the following notion. A \emph{filtration} of $F_N$ is a finite rooted tree, whose nodes are labelled by subgroups of $F_N$, in which the label of the root is $F_N$, and the children of a node labelled by $H$ are labelled by subgroups of $H$.

\begin{theo}\label{intro-FK}
Let $\mu$ be a probability measure on $\text{Out}(F_N)$, with finite first moment with respect to $d_{CV_N}^{sym}$. Then there exists a (deterministic) filtration of $F_N$, and a (deterministic) \emph{Lyapunov exponent} $\lambda_H\ge 0$ associated to each subgroup $H$ of the filtration, with $\lambda_{H'}\le\lambda_{H}$ as soon as $H'$ is a child of $H$, such that the following holds. 

For all $g\in F_N$ conjugate into a subgroup $H$ of the filtration, but not conjugate into any of the children of $H$, and almost every sample path $(\Phi_n)_{n\in\mathbb{N}}$ of the random walk on $(\text{Out}(F_N),\mu)$, we have $$\lim_{n\to +\infty}\frac{1}{n}\log ||\Phi_n(g)||=\lambda_H.$$ In addition, there are at most $\frac{3N-2}{4}$ positive Lyapunov exponents.
\end{theo}

We can actually be a bit more precise about the filtration that arises in Theorem \ref{intro-FK}, namely: at each level of the filtration, the children of $H$ coincide with a set of representatives of the conjugacy classes of point stabilizers of some very small $H$-tree with dense orbits. In addition, the conjugacy class of a subgroup of $F_N$ arising in the filtration has finite orbit under the subgroup of $\text{Out}(F_N)$ generated by the support of $\mu$.

To prove Theorem \ref{intro-FK}, we start by associating to almost every sample path of the random walk on $(\text{Out}(F_N),\mu)$ a (random) filtration, before showing that this filtration can actually be chosen to be deterministic (i.e. independent from the sample path). By Karlsson and Ledrappier's theorem, almost every sample path of the random walk is directed by a (random) horofunction. It follows from our description of the horoboundary of outer space that this horofunction is associated to a (random) tree $T$. We show that we may choose $T$ to have dense $F_N$-orbits as soon as the drift of the random walk is positive (if not, then all elements of $F_N$ have subexponential growth). The children of the root will be representatives of the conjugacy classes of the point stabilizers of $T$. We then argue by induction on the rank to construct the whole filtration. One ingredient of the proof is the study of stationary measures on the horoboundary of outer space.

A consequence of Theorem \ref{intro-FK} is that given an element $g\in F_N$, either $g$ grows subexponentially along almost every sample path of the random walk, or $g$ grows exponentially at speed $\lambda_H$ (independent of the sample path). The number of possible growth rates is uniformly bounded in the rank $N$ of the free group. 

The bound on the number of Lyapunov exponents was inspired by a result of Levitt about possible growth rates of elements of $F_N$ under iteration of a single automorphism \cite{Lev09}. In the case where $\mu$ is a Dirac measure supported on some element of $\text{Out}(F_N)$, Theorem \ref{intro-FK} specifies to Levitt's statement. In a sense, the bound on the number of Lyapunov exponents is optimal: in \cite{Lev09}, Levitt gave an example of a single automorphism of $F_N$ with exactly $\frac{3N-2}{4}$ exponential growth rates, see Example \ref{sharp-good} of the present paper. We saw however that in the opposite (and generic) case where the support of $\mu$ generates a sufficiently big subgroup of $\text{Out}(F_N)$, all conjugacy classes in $F_N$ have the same positive growth rate. 

By using Walsh's description of the horoboundary of the Teichmüller space for Thurston's metric, and building on Karlsson's ideas in \cite{Kar12}, our methods also yield the analogous result about growth of curves under random products of elements of the mapping class group of a compact surface. In this case, the appropriate analogue of the filtration is given by a decomposition of the surface into subsurfaces. 

We also provide a version of Theorem \ref{intro-random} in the case where increments are no longer assumed to be independent, in analogy to Karlsson's in \cite[Theorem 2]{Kar12}, see Section \ref{cocycle}. It would be interesting to know whether our methods can generalize to give a full version of an Oseledets-like theorem for ergodic cocycles of automorphisms of free groups.
\\
\\
\indent The paper is organised as follows. In Section \ref{sec-outer-space}, we recall basic facts about outer space, and present two ways of compactifying it (namely, Culler and Morgan's compactification, and the primitive compactification), as well as the Lipschitz (asymmetric) metric on outer space. Section \ref{sec-horo} is devoted to the identifiaction of the horofunction compactification of outer space with the primitive compactification. In Section \ref{sec-Busemann}, we investigate the geometry of the horoboundary of outer space. In particular, we discuss the link between the horoboundary and the metric completion of outer space, which was identified by Algom-Kfir in \cite{AK12} with the space of trees in $\overline{CV_N}$ having a nontrivial simplicial part, and with trivial arc stabilizers; we also identify the set of Busemann points with the set of trees with dense $F_N$-orbits in $\overline{CV_N}$. In Section \ref{sec-backward}, we discuss some properties of the horoboundary of outer space for the backward metric, and give a description when $N=2$. The last section of the paper is devoted to the study of the growth of elements of $F_N$ under random products of elements of $\text{Out}(F_N)$. We start with the case of ergodic cocycles of automorphisms (Section \ref{cocycle}), before turning to the case of random walks on $\text{Out}(F_N)$. The case of a random walk on a nonelementary subgroup is treated in Section \ref{sec-nonelementary}. Finally, Section \ref{sec-rw-2} is devoted to the proof of our $\text{Out}(F_N)$-version of Furstenberg--Kifer and Hennion's theorem. 

\section*{Acknowledgments}

I would like to thank my advisor Vincent Guirardel for suggesting many stimulating questions, and for the many discussions we had that led to improvements in the exposition of the material in the paper.

\section{Background on outer space} \label{sec-outer-space}

\subsection{Outer space}

Let $N\ge 2$. Denote by $R_N$ the graph having one vertex and $N$ edges, whose petals are identified with some free basis of $F_N$. A \emph{marked metric graph} is a pair $(X,\rho)$, where $X$ is a compact graph, all of whose vertices have valence at least $3$, equipped with a path metric (each edge being assigned a positive length that makes it isometric to a segment), and $\rho:R_N\to X$ is a homotopy equivalence. \emph{Outer space} $CV_N$ was defined by Culler and Vogtmann in \cite{CV86} to be the space of equivalence classes of marked metric graphs, two graphs $(X,\rho)$ and $(X',\rho')$ being equivalent if there exists a homothety $h:X\to X'$ such that $\rho'$ is homotopic to $h\circ\rho$. Passing to the universal cover, one can alternatively define outer space as the space of simplicial free, minimal, isometric actions of $F_N$ on simplicial metric trees, up to equivariant homothety (an action of $F_N$ on a tree is said to be \emph{minimal} if there is no proper invariant subtree). It is possible to normalize all the graphs in $CV_N$ to have volume $1$. We denote by $cv_N$ the \emph{unprojectivized outer space}, in which graphs (or equivalently trees) are considered up to isometry, instead of homothety.  The group $\text{Out}(F_N)$ acts on $CV_N$ on the right by precomposing the markings (we may also want to consider a left action by setting $\Phi.X:=X.\Phi^{-1}$ for $\Phi\in\text{Out}(F_N)$ and $X\in CV_N$). This action is proper but not cocompact, however outer space has a \emph{spine} $K_N$, which is a deformation retract of $CV_N$ on which $\text{Out}(F_N)$ acts cocompactly. For $\epsilon>0$, we also define the \emph{$\epsilon$-thick part} $CV_N^{\epsilon}$ of outer space to be the subspace of $CV_N$ consisting of (normalized) graphs having no loop of length smaller than $\epsilon$, on which $\text{Out}(F_N)$ acts cocompactly. The reader is referred to \cite{Vog02} for an excellent survey and reference article about outer space.

\subsection{Culler and Morgan's compactification of outer space} 

An \emph{$\mathbb{R}$-tree} is a metric space $(T,d_T)$ in which any two points $x$ and $y$ are joined by a unique embedded topological arc, which is isometric to a segment of length $d_T(x,y)$. Let $T$ be an \emph{$F_N$-tree}, i.e. an $\mathbb{R}$-tree equipped with an isometric action of $F_N$. For $g\in F_N$, the \emph{translation length} of $g$ in $T$ is defined to be $$||g||_T:=\inf_{x\in T}d_T(x,gx).$$ Culler and Morgan have shown in \cite[Theorem 3.7]{CM87} that the map
\begin{displaymath}
\begin{array}{cccc}
i:&cv_N&\to &\mathbb{R}^{F_N}\\
&T&\mapsto & (||g||_T)_{g\in F_N}
\end{array}
\end{displaymath}

\noindent is injective (and actually a homeomorphism onto its image for the \emph{weak topology} on outer space introduced in \cite{CV86}), whose image has projectively compact closure $\overline{CV_N}$ \cite[Theorem 4.5]{CM87}. Bestvina and Feighn \cite{BF94}, extending results by Cohen and Lustig \cite{CL95}, have characterized the points of this compactification, see also \cite{Hor14-5}. They showed that $\overline{CV_N}$ is the space of homothety classes of minimal, \emph{very small} $F_N$-trees, i.e. trees with trivial or maximally cyclic arc stabilizers and trivial tripod stabilizers. We also denote by $\overline{cv_N}$ the lift of $\overline{CV_N}$ to $\mathbb{R}^{F_N}$. We call the topology induced by this embedding on each of the spaces $CV_N, \overline{CV_N}, cv_N$ and $\overline{cv_N}$ the \emph{axes topology}, it is equivalent to the weak topology on $CV_N$. Bestvina and Feighn showed that $\overline{CV_N}$ has topological dimension $3N-4$, their result was improved by Gaboriau and Levitt who computed the dimension of the boundary $\partial CV_N$.

\begin{theo} \label{dimension} (Bestvina--Feighn \cite[Corollary 7.12]{BF94}, Gaboriau--Levitt \cite[Theorem V.2]{GL95})
The closure $\overline{CV_N}$ of outer space has dimension $3N-4$. The boundary $\partial CV_N$ has dimension $3N-5$.
\end{theo}

\subsection{Primitive compactification of outer space}\label{sec-prim}

Let $\mathcal{P}_N$ denote the set of primitive elements of $F_N$, i.e. elements that belong to some free basis of $F_N$. In \cite[Section 2.4]{Hor14-1}, we defined another compactification of outer space, called the \emph{primitive compactification}, by only looking at translation lengths of primitive elements of $F_N$. We get a continuous injective map $$i_{prim}:CV_N\to\mathbb{PR}^{\mathcal{P}_N}$$ which is a homeomorphism onto its image, and whose image has compact closure $\overline{CV_N}^{prim}$  \cite[Theorem 2.9]{Hor14-1}. This compactification is isomorphic to $\overline{CV_N}/\sim$, where $\sim$ denotes the \emph{primitive-equivalence} relation, that identifies two trees whose translation length functions are projectively equal in restriction to $\mathcal{P}_N$. The $\sim$-relation was explicitely described in \cite{Hor14-1}. In particular, we showed that the $\sim$-class of every tree with dense $F_N$-orbits is reduced to a point. We also proved that every $\sim$-class contains a \emph{canonical} representative $T$, so that for all trees $T'\sim T$, there is an $F_N$-equivariant morphism from $T$ to $T'$ (in particular, all elliptic elements in $T$ are also elliptic in $T'$).

The computation of the topological dimension of the closure and the boundary of outer space in \cite[Theorem V.2]{GL95} adapts to compute the topological dimension of $\overline{CV_N}^{prim}$ and the boundary $\partial CV_N^{prim}:=\overline{CV_N}^{prim}\smallsetminus CV_N$.

\begin{cor}
The space $\overline{CV_N}^{prim}$ has dimension $3N-4$. The boundary $\partial CV_N^{prim}$ has dimension $3N-5$.
\end{cor}

\begin{proof}
For all $T\in\overline{CV_N}^{prim}$, let $L(T)$ be the subgroup of $\mathbb{R}$ generated by the translation lengths in $T$ of all primitive elements of $F_N$. The \emph{$\mathbb{Q}$-rank} $r_{\mathbb{Q}}(T)$ is the dimension of the $\mathbb{Q}$-vector space $L(T)\otimes_{\mathbb{Z}}\mathbb{Q}$. Then \cite[Theorem IV.4]{GL95} shows that for all $T\in\overline{CV_N}^{prim}$, we have $r_{\mathbb{Q}}(T)\le 3N-3$, and that equality may hold only if $T\in CV_N$. In addition, we get as in \cite[Proposition V.1]{GL95} that the space $\mathbb{PR}_{\le k}^{\mathcal{P}_N}$ of all projectivized length functions with $\mathbb{Q}$-rank smaller than or equal to $k$ has topological dimension smaller than or equal to $k-1$. Since we can find a $(3N-4)$-simplex in $CV_N$, and a $(3N-5)$-simplex consisting of simplicial actions in $\partial CV_N^{prim}$, the claim follows.
\end{proof}

\subsection{Metric properties of outer space} \label{sec-metric}

There is a natural asymmetric metric on outer space, whose systematic study was initiated by Francaviglia and Martino in \cite{FM11}: given normalized marked metric graphs $(X,\rho)$ and $(X',\rho')$ in $CV_N$, the distance $d(X,X')$ is defined to be the logarithm of the infimal (in fact minimal by an easy Arzelà--Ascoli argument, see \cite[Lemma 3.4]{FM11}) Lipschitz constant of a map $f:X\to X'$ such that $\rho'$ is homotopic to $f\circ\rho$. This may also be defined as the logarithm of the infimal Lipschitz constant of an $F_N$-equivariant map between the corresponding trees (see the discussion in \cite[Sections 2.3 and 2.4]{AK12}). This defines a topology on outer space, which is equivalent to the classical one (see \cite[Theorems 4.11 and 4.18]{FM11}). The metric on outer space is not symmetric. One can define a symmetric metric by setting $d_{sym}(X,X'):=d(X,X')+d(X',X)$. Elements of $\text{Out}(F_N)$ act by isometries on $CV_N$ with respect to $d$ or $d_{sym}$. Given an $F_N$-tree $T\in CV_N$, an element $g\in F_N$ is a \emph{candidate} in $T$ if it is represented in the quotient graph $X:=T/F_N$ by a loop which is either
\begin{itemize}
\item an embedded circle in $X$, or

\item a bouquet of two circles in $X$, i.e. $\gamma=\gamma_1\gamma_2$, where $\gamma_1$ and $\gamma_2$ are embedded circles in $X$ which meet in a single point, or

\item a barbell graph, i.e. $\gamma=\gamma_1e\gamma_2\overline{e}$, where $\gamma_1$ and $\gamma_2$ are embedded circles in $X$ that do not meet, and $e$ is an embedded path in $X$ that meets $\gamma_1$ and $\gamma_2$ only at their origin (and $\overline{e}$ denotes the path $e$ crossed in the opposite direction). 
\end{itemize}

\noindent The following result, due to White, gives an alternative description of the metric on outer space. A proof can be found in \cite[Proposition 3.15]{FM11}, it was simplified by Algom-Kfir in \cite[Proposition 2.3]{AK11}.

\begin{theo} \label{White} (White, see \cite[Proposition 3.15]{FM11} or \cite[Proposition 2.3]{AK11}) 
For all $T$,$T'\in CV_N$, we have $$d(T,T')=\log \sup_{g\in F_N\smallsetminus\{e\}}\frac{||g||_{T'}}{||g||_T}.$$ The supremum is achieved for an element $g\in F_N$ which is a candidate in $T$.
\end{theo} 

Notice in particular that candidates in $T$ are primitive elements of $F_N$ (see \cite[Lemma 1.12]{Hor14-1}, for instance). White's theorem has been extended by Algom-Kfir in \cite[Proposition 4.5]{AK12} to the case where $T\in\overline{cv_N}$ is a simplicial tree, and $T'\in\overline{cv_N}$ is arbitrary (Algom-Kfir actually states her result for trees in the metric completion of outer space). The extension to all trees in $\overline{cv_N}$ was made in \cite[Theorem 0.3]{Hor14-1}. Given $T,T'\in\overline{cv_N}$, we denote by $\text{Lip}(T,T')$ the infimal Lipschitz constant of an $F_N$-equivariant map from $T$ to the metric completion $\overline{T'}$ if such a map exists, and $+\infty$ otherwise. In the following statement, we take the conventions $\frac{1}{0}=+\infty$ and $\frac{0}{0}=0$.

\begin{theo}(Horbez \cite[Theorem 0.3]{Hor14-1})
For all $T,T'\in\overline{cv_N}$, we have $$\text{Lip}(T,T')=\sup_{g\in F_N}\frac{||g||_{T'}}{||g||_T}.$$
\end{theo} 

\section{The horocompactification of outer space} \label{sec-horo}

In this section, we define the horocompactification of outer space and show that it is isomorphic to the primitive compactification (in particular, it has finite topological dimension). Our approach is motivated by Walsh's analogous statements in the case of the Teichmüller space of a surface, equipped with Thurston's asymmetric metric \cite{Wal11}. 

We start by recalling the construction of a compactification of an (asymmetric) metric space by horofunctions, under some geometric assumptions. This notion was first introduced in the symmetric case by Gromov in \cite{Gro81}, we refer the reader to \cite[Section 2]{Wal11} for the case of an asymmetric metric.

Let $(X,d)$ be a (possibly asymmetric) metric space, and let $b\in X$ be some fixed basepoint. For all $z\in X$, we define a continuous map 
\begin{displaymath}
\begin{array}{cccc}
\psi_z:&X&\to&\mathbb{R}\\
&x&\mapsto &d(x,z)- d(b,z)
\end{array}.
\end{displaymath}

\noindent Let $\mathcal{C}(X)$ be the space of continuous real-valued functions on $X$, equipped with the topology of uniform convergence on compact sets of $(X,d_{sym})$ (where we recall that $d_{sym}(x,y):=d(x,y)+d(y,x)$). We get a map

\begin{center}
$\begin{array}{cccc}
\psi :& X&\to &\mathcal{C}(X)\\
& z&\mapsto& \psi_z
\end{array}$
\end{center}

\noindent which is continuous and injective, see \cite[Chapter II.1]{Bal95} or \cite[Lemma 2.1]{Wal11}. We say that an asymmetric metric space is \emph{quasi-proper} if
\begin{itemize}
\item the space $(X,d)$ is geodesic, and
\item the space $(X,d_{sym})$ is proper (i.e. closed balls are compact), and 
\item for all $x\in X$ and all sequences $(x_n)_{n\in\mathbb{N}}$ of elements of $X$, the distance $d(x_n,x)$ converges to $0$ if and only if $d(x,x_n)$ does.
\end{itemize} 

\begin{prop} \label{horocompactification-general}(Ballmann \cite[Chapter II.1]{Bal95}, Walsh \cite[Proposition 2.2]{Wal11})
Let $(X,d)$ be a (possibly asymmetric) quasi-proper metric space. Then $\psi$ defines a homeomorphism from $X$ to its image in $\mathcal{C}(X)$, and the closure $\overline{\psi(X)}$ in $\mathcal{C}(X)$ is compact.
\end{prop}

We call $\overline{\psi(X)}$ the \emph{horocompactification} of $X$, the elements in $X(\infty):=\overline{\psi(X)}\smallsetminus\psi(X)$ being \emph{horofunctions}. As noted in \cite[Section 2]{Wal11}, all the functions in $\overline{\psi(X)}$ are $1$-Lipschitz with respect to $d_{sym}$, so uniform convergence on compact sets of $(X,d_{sym})$ is equivalent to pointwise convergence. By the work of Francaviglia and Martino \cite[Theorems 5.5, 4.12 and 4.18]{FM11}, outer space is quasi-proper, so we can define its horocompactification. 

\begin{theo} \label{horocompactification}
There exists a unique $\text{Out}(F_N)$-equivariant homeomorphism from $\overline{CV_N}^{prim}$ to the horocompactification of $CV_N$ which restricts to the identity on $CV_N$. For all $z\in\overline{CV_N}^{prim}$, the horofunction associated to $z$ is given by $$\psi_z(x)= \log \sup_{g\in \mathcal{P}_N}\frac{||g||_z}{||g||_x}-\log\sup_{g\in \mathcal{P}_N}\frac{||g||_z}{||g||_b}$$ for all $x\in CV_N$ (identified with its covolume $1$ representative).
\end{theo}

\begin{proof}
Uniqueness follows from the density of $CV_N$ in $\overline{CV_N}^{prim}$. Let $x\in CV_N$, which we identify with its  covolume $1$ representative. For all $z\in \overline{CV_N}^{prim}$, we let $$\psi'_z(x):=\log \sup_{g\in \mathcal{P}_N}\frac{||g||_z}{||g||_x}-\log\sup_{g\in \mathcal{P}_N}\frac{||g||_z}{||g||_b}.$$ This is well-defined, because $\psi'_z$ only depends on the projective class of $||.||_z$. By definition of the metric on $CV_N$, we have $\psi'_z=\psi_z$ for all $z\in CV_N$. In addition, the suprema arising in the expression of $\psi_z'(x)$ are achieved on finite sets $\mathcal{F}(x)$ (resp. $\mathcal{F}(b)$) consisting of candidates in $x$ (resp. in $b$) by Theorem \ref{White}. 

We claim that for all $z\in\overline{CV_N}^{prim}$, the map $\psi'_z$ is continuous. Indeed, let $z\in\overline{CV_N}^{prim}$, and let $(z_n)_{n\in\mathbb{N}}\in{CV_N}^{\mathbb{N}}$ be a sequence that converges to $z$. For all $n\in\mathbb{N}$, we have $$\psi'_{z_n}(x)=\log \sup_{\mathcal{F}(x)}\frac{||g||_{z_n}}{||g||_x}-\log\sup_{\mathcal{F}(b)}\frac{||g||_{z_n}}{||g||_b}.$$ By definition of the topology on $\mathbb{PR}^{\mathcal{P}_N}$, there exists a sequence $(\lambda_n)_{n\in\mathbb{N}}$ of real numbers such that for all $g\in \mathcal{P}_N$, the sequence $(\lambda_n||g||_{z_n})_{n\in\mathbb{N}}$ converges to $||g||_z$. So $\psi'_{z_n}(x)$ converges to $\psi'_{z}(x)$. Therefore, the map $\psi'_z$ is the pointwise limit of the $1$-Lipschitz maps $\psi'_{z_n}$, so $\psi'_z$ is continuous.
 
We can thus extend the map $\psi$ to a map from $\overline{CV_N}^{prim}$ to $\mathcal{C}(CV_N)$, which we still denote by $\psi$. We claim that this extension is continuous. Indeed, if a sequence $(z_n)_{n\in\mathbb{N}}\in (\overline{CV_N}^{prim})^{\mathbb{N}}$ converges to $z\in\overline{CV_N}^{prim}$, then the maps $\psi_{z_n}$ converge pointwise to $\psi_z$, and hence they converge uniformly on compact sets of $(X,d_{sym})$ because all maps $\psi_{z_n}$ are $1$-Lipschitz.

We now prove that the map $\psi:\overline{CV_N}^{prim}\to\mathcal{C}(CV_N)$ is injective. Let $z$,$z'\in\overline{CV_N}^{prim}$ be such that $\psi_z=\psi_{z'}$. Let $g\in \mathcal{P}_N$. Let $x\in CV_N$ be a rose, one of whose petals is labelled by $g$. Denote by $x_{\epsilon}$ the rose in $CV_N$ with same underlying graph as $x$, in which the petal labelled by $g$ has length $\epsilon>0$, while the other petals all have the same length. As $\epsilon$ tends to $0$, the length $||g||_{x_{\epsilon}}$ tends to $0$, while $||g'||_{x_{\epsilon}}$ is bounded below for all $\epsilon>0$ and all $g'\neq g^{\pm 1}\in \mathcal{F}(x_{\epsilon})$, and $\mathcal{F}(x_{\epsilon})$ does not depend on $\epsilon$. Hence for $\epsilon>0$ sufficiently small, we have the following dichotomy (we fix representatives of $z$ and $z'$ in their projective classes).

\begin{itemize}
\item If $||g||_z\neq 0$, then $\psi_z(x_{\epsilon})=\log\frac{||g||_z}{\epsilon C(z)}$ (where $C(z):=\sup_{\mathcal{F}(b)}\frac{||g||_z}{||g||_b}$) tends to $+\infty$ as $\epsilon$ goes to $0$.
\item If $||g||_z=0$, then $\psi_{z}(x_{\epsilon})$ is bounded above independently of $\epsilon>0$.
\end{itemize}

\noindent As $\psi_z=\psi_{z'}$, an element $g\in\mathcal{P}_N$ is elliptic in $z$ if and only if it is elliptic in $z'$, and in addition, we have $\frac{||g||_{z'}}{||g||_z}=\frac{C(z)}{C(z')}$ for all $g\in \mathcal{P}_N$ which are not elliptic in $z$. Hence $z=z'$.

We have shown that $\psi:\overline{CV_N}^{prim}\to\mathcal{C}(CV_N)$ is a continuous injection. As $\overline{CV_N}^{prim}$ is compact, the map $\psi$ is a homeomorphism from $\overline{CV_N}^{prim}$ to its image in $\mathcal{C}(CV_N)$. In particular, the image $\psi(\overline{CV_N}^{prim})$ is compact, and hence closed in $\mathcal{C}(CV_N)$. By continuity of $\psi$, we also have $\psi(CV_N)\subseteq\psi(\overline{CV_N}^{prim})\subseteq\overline{\psi(CV_N)}$, so $\psi(\overline{CV_N}^{prim})=\overline{\psi(CV_N)}$, i.e. $\overline{CV_N}^{prim}$ is isomorphic to the horocompactification of $CV_N$. That $\psi$ is $\text{Out}(F_N)$-equivariant follows from its construction.
\end{proof}

\begin{rk} \label{rk-horo-spine}
In order to prove the injectivity of the map $\psi:\overline{CV_N}^{prim}\to\mathcal{C}(CV_N)$, we "select" the primitive element $g$ in the rose $x$ by making the length of the corresponding petal tend to $0$. There is another way of "selecting" the primitive element $g$ which does not require leaving the thick part of outer space, and will therefore enable us to prove the corresponding statement for the spine or the thick part of outer space. The idea is to replace the rose $x$ whose petals all have the same length, and are labelled by a basis $(g,g_2,\dots,g_N)$ of $F_N$, by a rose $x_k$ whose petals are labelled by $(g,g_2,\dots,g_{N-1},g_Ng^k)$, for $k$ sufficiently large. Unless $||g||_z=0$, the translation length in $z$ of a word represented by a candidate in $x_k$ containing the petal labelled by $g_Ng^k$ becomes arbitrarily large as $k$ tends to $+\infty$, and the translation lengths of such a word in two (unprojectivized) trees $z$ and $z'$ may be equal for arbitrarily large $k$ only if $||g||_z=||g||_{z'}$. 
\end{rk}

The spine $K_N$ (considered as a subspace of $CV_N$) and the $\epsilon$-thick part $CV_N^{\epsilon}$, equipped with the restriction of the Lipschitz metric, are not geodesic metric spaces. However, we show that we can still define their horocompactification. We recall that for all metric spaces $X$, we have defined an embedding $\psi:X\to\mathcal{C}(X)$. We define $\overline{K_N}^{prim}$ and $\overline{CV_N^{\epsilon}}^{prim}$ in the same way as we defined $\overline{CV_N}^{prim}$.

\begin{prop}
The map $\psi$ defines a homeomorphism from $K_N$ to its image in $\mathcal{C}(K_N)$, and the closure $\overline{\psi(K_N)}$ in $\mathcal{C}(K_N)$ is compact, and isomorphic to $\overline{K_N}^{prim}$. For all $\epsilon>0$, the map $\psi$ defines a homeomorphism from $CV_N^{\epsilon}$ to its image in $\mathcal{C}(CV_N^{\epsilon})$, and the closure $\overline{\psi(CV_N^{\epsilon})}$ in $\mathcal{C}(CV_N^{\epsilon})$ is compact, and isomorphic to $\overline{CV_N^{\epsilon}}^{prim}$.
\end{prop}

\begin{proof}
In the proof of Proposition \ref{horocompactification-general}, the assumption that $(X,d)$ is geodesic is only used to show that if $(z_n)_{n\in\mathbb{N}}$ is a sequence in $X$ escaping to infinity (i.e. eventually leaving and never returning to every compact set), then no subsequence of $(\psi_{z_n})_{n\in\mathbb{N}}$ converges to a function $\psi_y$ with $y\in X$.

Assume that there exists a sequence $(z_n)_{n\in\mathbb{N}}$ of elements of $K_N$ escaping to infinity such that some subsequence of $(\psi_{z_n})_{n\in\mathbb{N}}$ converges to $\psi_y$, with $y\in K_N$. Up to passing to a subsequence again, we may assume that $(z_n)_{n\in\mathbb{N}}$ converges to an element $z$ in $\overline{CV_N}^{prim}$ (and actually $z\in\partial CV_N^{prim}$), so by Theorem \ref{horocompactification} we have $\psi_z=\psi_y$. However, in this case, the argument in Remark \ref{rk-horo-spine} shows that $z=y$, a contradiction. So $\psi$ defines a homeomorphism from $K_N$ to its image in $\mathcal{C}(K_N)$, and the closure $\overline{\psi(K_N)}$ in $\mathcal{C}(K_N)$ is compact. The argument then goes as in the proof of Theorem \ref{horocompactification}, by using Remark \ref{rk-horo-spine}, to show that $\overline{\psi(K_N)}$ is isomorphic to $\overline{K_N}^{prim}$. The same argument also yields the result for the $\epsilon$-thick part of outer space.
\end{proof}

\section{Completion and Busemann points}\label{sec-Busemann}

\subsection{The metric completion of outer space}

We follow Algom-Kfir's exposition in \cite[Section 1]{AK12} of the construction of a completion of an asymmetric metric space. A sequence $(x_n)_{n\in\mathbb{N}}$ of elements in a (possibly asymmetric) metric space $X$ is \emph{forward admissible} if for all $\epsilon>0$, there exists $N(\epsilon)\in\mathbb{N}$ such that for all $n\ge N(\epsilon)$, there exists $K(n,\epsilon)\in\mathbb{N}$ such that $d(x_n,x_k)\le\epsilon$ for all $k\ge K(n,\epsilon)$. Two forward admissible sequences are \emph{equivalent} if their \emph{interlace} (i.e. the sequence $(z_n)_{n\in\mathbb{N}}$ defined by $z_{2n}=x_n$ and $z_{2n+1}=y_n$ for all $n\in\mathbb{N}$) is forward admissible. The \emph{forward metric completion} $\widehat{X}$ of $X$ is defined to be the set of equivalence classes of forward admissible sequences. The reader is referred to \cite[Section 1]{AK12} for a detailed account of this construction.

\begin{lemma}\label{c} (Algom-Kfir \cite[Lemma 1.8]{AK12})
Let $X$ be a (possibly asymmetric) metric space, and let $(x_n)_{n\in\mathbb{N}}$ and $(y_n)_{n\in\mathbb{N}}$ be two forward admissible sequences of elements in $X$, then either
\begin{itemize}
\item for all $r\ge0$, there exists $N(r)\in\mathbb{N}$ such that for all $n\ge N(r)$, there exists $K(n,r)\in\mathbb{N}$ such that for all $k\ge K(n,r)$, we have $d(x_n,y_k)\ge r$, or
\item there exists $c\ge 0$ such that for all $\epsilon>0$, there exists $N(\epsilon)\in\mathbb{N}$ such that for all $n\ge N(\epsilon)$, there exists $K(n,\epsilon)\in\mathbb{N}$ such that for all $k\ge K(n,\epsilon)$, we have $|d(x_n,y_k)-c|\le \epsilon$.
\end{itemize}
\end{lemma}

Given two forward admissible sequences $(x_n)_{n\in\mathbb{N}}$ and $(y_n)_{n\in\mathbb{N}}$ of elements in $X$, we denote by $c((x_n),(y_n))$ the number provided by Lemma \ref{c} (in the first case, we set $c((x_n),(y_n)):=+\infty$). In the particular case where $(x_n)_{n\in\mathbb{N}}$ is constant, Algom-Kfir's proof of Lemma \ref{c} actually shows that the second case occurs.

\begin{lemma}\label{c(b,.)}
For all $b\in X$, and all forward admissible sequences $(z_n)_{n\in\mathbb{N}}\in X^{\mathbb{N}}$, we have $c(b,(z_n))<+\infty$. If $(z_n)_{n\in\mathbb{N}}\in X^{\mathbb{N}}$ and $(z'_n)_{n\in\mathbb{N}}\in X^{\mathbb{N}}$ are two equivalent forward admissible sequences, then $c(b,(z_n))=c(b,(z'_n))$.
\end{lemma}

\begin{proof}
As $(z_n)_{n\in\mathbb{N}}$ is forward admissible, the sequence $(d(b,z_n))_{n\in\mathbb{N}}$ is almost monotonically decreasing in the sense of \cite[Definition 1.9]{AK12}. Hence by \cite[Proposition 1.10]{AK12}, it converges to a limit $c(b,(z_n))$. If $(z_n)_{n\in\mathbb{N}}\in X^{\mathbb{N}}$ and $(z'_n)_{n\in\mathbb{N}}\in X^{\mathbb{N}}$ are two equivalent forward admissible sequences, then $c(b,(z_n))=c(b,(z'_n))=c(b,(z''_n))$, where $(z''_n)_{n\in\mathbb{N}}$ is the interlace of $(z_n)_{n\in\mathbb{N}}$ and $(z'_n)_{n\in\mathbb{N}}$.
\end{proof}

Algom-Kfir shows that two forward admissible sequences $(x_n)_{n\in\mathbb{N}}$ and $(y_n)_{n\in\mathbb{N}}$ of elements in $X$ are equivalent if and only if $c((x_n),(y_n))=c((y_n),(x_n))=0$ \cite[Lemma 1.12]{AK12}. The metric on $X$ extends to an asymmetric metric $\widehat{d}$ on $\widehat{X}$ (which might not satisfy the separation axiom, and might be $\infty$-valued) by setting $\widehat{d}((x_n),(y_n)):=c((x_n),(y_n))$ \cite[Proposition 1.16]{AK12}. The collection of balls $B(x,r):=\{y\in\widehat{X}|\widehat{d}(y,x)<r\}$ for $x\in\widehat{X}$ and $r\in\mathbb{R}_+^{\ast}$ is a basis for a topology on $\widehat{X}$. One can also consider the symmetrized metric $\widehat{d}_{sym}$ on $\widehat{X}$, which defines another topology on $\widehat{X}$. 
\\
\\
\indent Algom-Kfir has determined the metric completion of outer space in \cite{AK12}.

\begin{theo}\label{AK} (Algom-Kfir \cite[Theorem B]{AK12})
Let $T\in\overline{CV_N}$. Then $T\in\widehat{CV_N}$ if and only if $T$ does not have dense orbits, and $T$ has trivial arc stabilizers. In addition, for all $T,T'\in\widehat{CV_N}$, we have $\widehat{d}(T,T')=\log\text{Lip}(\widetilde{T},\widetilde{T'})$, where $\widetilde{T}$ (resp. $\widetilde{T'}$) denotes the covolume one representative of $T$ (resp. $T'$) in $\overline{cv_N}$. 
\end{theo}

\subsection{The metric completion as a subspace of the horocompactification}

Throughout the section, we assume that $X$ is a quasi-proper metric space, so that the horocompactification of $X$ is well-defined. We recall that associated to any $z\in X$ is a function $\psi_z\in\mathcal{C}(X)$.

\begin{prop}\label{completion-embeds-1}
For all forward admissible sequences $(z_n)_{n\in\mathbb{N}}\in X^{\mathbb{N}}$, the sequence $(\psi_{z_n})_{n\in\mathbb{N}}$ has a limit in $\mathcal{C}(X)$. If $(z_n)_{n\in\mathbb{N}}$ and $(z'_n)_{n\in\mathbb{N}}$ are two equivalent forward admissible sequences, then the sequences $(\psi_{z_n})_{n\in\mathbb{N}}$ and $(\psi_{z'_n})_{n\in\mathbb{N}}$ converge to the same limit in $\mathcal{C}(X)$.
\end{prop}

\begin{proof}
Let $(z_n)_{n\in\mathbb{N}}$ and $(z'_n)_{n\in\mathbb{N}}$ be two equivalent forward admissible sequences. Let $(z_{\sigma(n)})_{n\in\mathbb{N}}$ (resp. $(z'_{\sigma'(n)})_{n\in\mathbb{N}}$) be a subsequence of $(z_n)_{n\in\mathbb{N}}$ (resp. $(z'_n)_{n\in\mathbb{N}}$) that converges to some function $\psi$ (resp. $\psi'$) in $\mathcal{C}(X)$. Let $\epsilon>0$. By definition of $c$, there exists an integer $N(\epsilon)\in\mathbb{N}$ such that for all $n\ge N(\epsilon)$, there exists $K(n,\epsilon)\in\mathbb{N}$ such that for all $k\ge K(n,\epsilon)$, we have $d(z_{\sigma(n)},z'_{\sigma'(k)})\le\epsilon$. In addition, Lemma \ref{c(b,.)} shows the existence of $N'(\epsilon)\in\mathbb{N}$ and $c\in\mathbb{R}$ such that for all $n\ge N'(\epsilon)$, we have $|d(b,z_{\sigma(n)})-c|\le\epsilon$ and $|d(b,z'_{\sigma'(n)})-c|\le\epsilon$. For all  $n\ge\max\{N(\epsilon),N'(\epsilon)\}$, all $k\ge \max\{K(n,\epsilon),N'(\epsilon)\}$ and all $x\in X$, we have 
\begin{displaymath}
\begin{array}{rl}
\psi_{z'_{\sigma'(k)}}(x)-\psi_{z_{\sigma(n)}}(x)&=d(x,z'_{\sigma'(k)})-d(x,z_{\sigma(n)})+d(b,z_{\sigma(n)})-d(b,z'_{\sigma'(k)}) \\
&\le d(z_{\sigma(n)},z'_{\sigma'(k)})+d(b,z_{\sigma(n)})-d(b,z'_{\sigma'(k)})\\
&\le 3\epsilon.
\end{array}
\end{displaymath}

\noindent By making $\epsilon>0$ arbitrarily small, and letting $n$ and $k$ tend to infinity, we get that $\psi'(x)\le\psi(x)$ for all $x\in X$. Symmetrizing the argument, we also get that $\psi(x)\le\psi'(x)$ for all $x\in X$, whence $\psi=\psi'$. In particular, the sequence $(\psi_{z_n})_{n\in\mathbb{N}}$ associated to any forward admissible sequence $(z_n)_{n\in\mathbb{N}}$ has at most one limit point, and hence it converges in the horocompactification of $X$. Two equivalent sequences give rise to the same limit.
\end{proof}

Proposition \ref{completion-embeds-1} yields a map $i$ from the metric completion $\widehat{X}$ of $X$ to the horocompactification of $X$, which is the identity map in restriction to $X$, by setting
\begin{displaymath}
\begin{array}{cccc}
i:&\widehat{X}&\to &X\cup X(\infty)\\
&(z_n)_{n\in\mathbb{N}}&\mapsto & \lim_{n\to +\infty}\psi_{z_n}
\end{array}.
\end{displaymath}

\begin{prop}
The map $i:\widehat{X}\to X\cup X(\infty)$ is injective.
\end{prop}

\begin{proof}
Let $(z_n)_{n\in\mathbb{N}},(z'_n)_{n\in\mathbb{N}}\in X^{\mathbb{N}}$ be two forward admissible sequences. Assume that $i((z_n)_{n\in\mathbb{N}})=i((z'_n)_{n\in\mathbb{N}})=\psi\in\mathcal{C}(X)$. Let $\epsilon>0$. For all $p\in\mathbb{N}$, there exists $K_0(p,\epsilon)$ such that for all $n,q\ge K_0(p,\epsilon)$, we have $|\psi_{z_n}(z_p)-\psi_{z'_q}(z_p)|\le\epsilon$ and $|\psi_{z_n}(z'_p)-\psi_{z'_q}(z'_p)|\le\epsilon$. As $(z_n)_{n\in\mathbb{N}}$ is forward admissible, there exists $N_1(\epsilon)\in\mathbb{N}$ such that for all $n\ge N_1(\epsilon)$, there exists $K_1(n,\epsilon)\in\mathbb{N}$ such that for all $k\ge K_1(n,\epsilon)$, we have $d(z_n,z_k)\le\epsilon$. As $(z'_n)_{n\in\mathbb{N}}$ is forward admissible, there exists $N_2(\epsilon)\in\mathbb{N}$ such that for all $n\ge N_2(\epsilon)$, there exists $K_2(n,\epsilon)\in\mathbb{N}$ such that for all $k\ge K_2(n,\epsilon)$, we have $d(z'_n,z'_k)\le\epsilon$. By Lemma \ref{c(b,.)}, there also exist $N_3(\epsilon)\in\mathbb{N}$ and $c,c'\in\mathbb{R}$ such that for all $n\ge N_3(\epsilon)$, we have $|d(b,z_n)-c|\le\epsilon$ and $|d(b,z'_n)-c'|\le\epsilon$.  For all $p\ge N_1(\epsilon)$, all $n\ge\max\{K_0(p,\epsilon),K_1(p,\epsilon),N_3(\epsilon)\}$ and all $q\ge \max\{K_0(p,\epsilon),N_3(\epsilon)\}$, as $$\psi_{z_n}(z_p)-\psi_{z'_q}(z_p)=d(z_p,z_n)-d(z_p,z'_q)+d(b,z'_q)-d(b,z_n),$$ we have $d(z_p,z'_q)+c-c'\le 4\epsilon$. In particular, for all $\epsilon>0$, we have $c-c'\le 4\epsilon$, whence $c\le c'$. Similarly, for all $p\ge N_2(\epsilon)$, all $n\ge \max\{K_0(p,\epsilon),N_3(\epsilon)\}$ and all $q\ge\max\{K_0(p,\epsilon),K_2(p,\epsilon),N_3(\epsilon)\}$, as $$\psi_{z_n}(z'_p)-\psi_{z'_q}(z'_p)=d(z'_p,z_n)-d(z'_p,z'_q)+d(b,z'_q)-d(b,z_n),$$ we have $d(z'_p,z_n)+c'-c\le 4\epsilon$, and in particular this implies that $c'\le c$. So $c=c'$, and the inequalities we have established thus imply that $c((z_n),(z'_n))=c((z'_n),(z_n))=0$. It then follows from \cite[Lemma 1.12]{AK12} that the sequences $(z_n)_{n\in\mathbb{N}}$ and $(z'_n)_{n\in\mathbb{N}}$ are equivalent, thus showing that $i$ is injective.
\end{proof}

\noindent In particular, the space $\widehat{X}$ inherits a (metrizable) topology induced by the topology on $\mathcal{C}(X)$. 

\subsection{Comparing the topologies on $\widehat{X}$}

Let $X$ be a quasi-proper (possibly asymmetric) metric space. We now compare the three topologies on $\widehat{X}$ we have introduced in the previous two sections, namely the topology defined by $\widehat{d}_{sym}$, the topology defined by $d$, and the topology coming from $\mathcal{C}(X)$. The topology defined by $\widehat{d}_{sym}$ dominates the topology defined by $\widehat{d}$. The following proposition shows that the topology defined by $\widehat{d}_{sym}$ also dominates the topology induced by the topology coming from $\mathcal{C}(X)$ (all these topologies are second-countable, which justifies the use of sequential arguments).

\begin{prop}\label{dsym-horo}
Let $z=(z_n)_{n\in\mathbb{N}}\in\widehat{X}$, and let $(z^k)_{k\in\mathbb{N}}=((z_n^k)_{n\in\mathbb{N}})_{k\in\mathbb{N}}$ be a sequence of elements of $\widehat{X}$. If $\widehat{d}_{sym}(z^k,z)$ converges to $0$, then $\psi_{z^k}$ converges to $\psi_z$ in $\mathcal{C}(X)$.
\end{prop}

\begin{proof}
Assume that $\widehat{d}_{sym}(z^k,z)$ converges to $0$, i.e. $c(z^k,z)$ and $c(z,z^k)$ both converge to $0$. Let $\epsilon>0$. There exists $K_0\in\mathbb{N}$ such that for all $k\ge K_0$, we have $c(z^k,z)\le\epsilon$ and $c(z,z^k)\le\epsilon$. We fix $k\ge K_0$. As $c(z^k,z)\le\epsilon$, there exists an integer $N_1(\epsilon)\in\mathbb{N}$ such that for all $n\ge N_1(\epsilon)$, there exists $K_1(n,\epsilon)\in\mathbb{N}$ such that for all $m\ge K_1(n,\epsilon)$, we have $d(z_n^k,z_{m})\le 2\epsilon$. Similarly, as $c(z,z^k)\le\epsilon$, there exists $N_2(\epsilon)\in\mathbb{N}$ such that for all $n\ge N_2(\epsilon)$, there exists $K_2(n,\epsilon)\in\mathbb{N}$ such that for all $m\ge K_2(n,\epsilon)$, we have $d(z_n,z_m^k)\le 2\epsilon$. By Lemma \ref{c(b,.)}, there also exists $N'(\epsilon)\in\mathbb{N}$, and $c^k,c\in\mathbb{R}$ such that for all $n\ge N'(\epsilon)$, we have $|d(b,z_n^k)-c^k|\le\epsilon$ and $|d(b,z_n)-c|\le\epsilon$. Choosing $n\ge\max\{N_1(\epsilon),N'(\epsilon)\}$ and $m\ge\max\{ K_1(n,\epsilon),N'(\epsilon)\}$, we get that
\begin{displaymath}
\begin{array}{rl}
c-c^k &\le d(b,z_m)-d(b,z_n^k)+2\epsilon\\
&\le d(z_n^k,z_m)+2\epsilon\\
&\le 4\epsilon. 
\end{array}
\end{displaymath}

\noindent So for all $n\ge\max\{N_2(\epsilon),N'(\epsilon)\}$, all $m\ge\max\{K_2(n,\epsilon),N'(\epsilon)\}$ and all $x\in X$, we have 
\begin{displaymath}
\begin{array}{rl}
\psi_{z_m^k}(x)-\psi_{z_n}(x)&= d(x,z_m^k)-d(x,z_n)+d(b,z_n)-d(b,z_m^k)\\
&\le d(z_n,z_m^k)+c-c^k+2\epsilon\\
&\le 8\epsilon.
\end{array}
\end{displaymath}

\noindent Letting $m$ go to infinity, we get that $\psi_{z^k}(x)-\psi_{z_n}(x)\le 8\epsilon$, and letting $n$ go to infinity, we get $\psi_{z^k}(x)-\psi_{z}(x)\le 8\epsilon$. Similarly, choosing $n\ge\max\{N_2(\epsilon),N'(\epsilon)\}$ and $m\ge\max\{K_2(n,\epsilon),N'(\epsilon)\}$, we get that
\begin{displaymath}
\begin{array}{rl}
c^k-c&\le d(b,z_m^k)-d(b,z_n)+2\epsilon\\
&\le d(z_n,z_m^k)+2\epsilon\\
&\le 4\epsilon.
\end{array}
\end{displaymath}

\noindent So for all $n\ge\max\{N_1(\epsilon),N'(\epsilon)\}$, all $m\ge\max\{K_1(n,\epsilon),N'(\epsilon)\}$ and all $x\in X$, we have
\begin{displaymath}
\begin{array}{rl}
\psi_{z_m}(x)-\psi_{z_n^k}(x)&=d(x,z_m)-d(x,z_n^k)+d(b,z_n^k)-d(b,z_m)\\
&\le d(z_n^k,z_m)+c^k-c+2\epsilon\\
&\le 8\epsilon.
\end{array}
\end{displaymath}

\noindent Again letting $m$ and then $n$ tend to infinity, we get that $\psi_{z}(x)-\psi_{z^k}(x)\le 8\epsilon$ for all $x\in X$. So $|\psi_{z^k}(x)-\psi_{z}(x)|\le 8\epsilon$ for all $x\in X$ and all $k\ge K_0$. Hence $(\psi_{z^k})_{k\in\mathbb{N}}$ converges uniformly (and in particular uniformly on compact sets) to $\psi_z$.
\end{proof}

However, the examples below show that no two of the three topologies we have defined on $\widehat{X}$ are equivalent when $X=CV_N$. In this case, the topology induced by the topology on $\mathcal{C}(X)$ is the \emph{primitive axes topology}, given by the embedding of $\widehat{X}$ into $\mathbb{PR}^{\mathcal{P}_N}$. In the case of outer space, there is a fourth natural topology on $\widehat{CV_N}$, called the \emph{axes topology}, given by the embedding into $\mathbb{PR}^{F_N}$. The axes topology dominates the primitive axes topology. The examples below show that no two of the four topologies on $\widehat{CV_N}$ are equivalent. 

\begin{ex} \emph{The topology defined by $\widehat{d}_{sym}$ is not dominated by any of the other three topologies.}\\
Let $T\in\widehat{CV_N}$ be a nonsimplicial tree. Let $T^{simpl}\in\widehat{CV_N}$ be the tree obtained by collapsing the nonsimplicial part of $T$. For $n\in\mathbb{N}$, let $T_n$ be the tree obtained from $T$ by applying a homothety with factor $\frac{1}{n}$ to all vertex trees of $T$. Then the sequence $(T_n)_{n\in\mathbb{N}}$ converges to $T^{simpl}\in\widehat{CV_N}$ in the axes topology (and hence also in the primitive axes topology). For all $n\in\mathbb{N}$, there is an obvious $F_N$-equivariant $1$-Lipschitz map from $T_n$ to $T^{simpl}$ given by collapsing all components of the complement of the simplicial part of $T$ to points, so $\widehat{d}(T_n,T^{simpl})=0$, while $\widehat{d}(T^{simpl},T_n)=+\infty$. So the sequence $(T_n)_{n\in\mathbb{N}}$ also converges to $T^{simpl}$ in the topology defined by $\widehat{d}$, but not in the topology defined by $\widehat{d}_{sym}$.  
\end{ex}

\begin{ex} \emph{The topology defined by $\widehat{d}$ does not dominate the primitive axes topology.}\\
Let $T\in\widehat{CV_N}$ be a nonsimplicial tree. As in the previous example, we have $\widehat{d}(T,T^{simpl})=0$. Hence the space $(\widehat{CV_N},\widehat{d})$ is not separated, while $\mathcal{C}(CV_N)$ is. 
\end{ex}

\begin{ex} \emph{The axes topology does not dominate the topology defined by $\widehat{d}$.}\\
Let $T'\in\overline{cv_N}$ be a tree with dense orbits, and let $p\in T'$. Let $(T'_n,p_n)_{n\in\mathbb{N}}$ be a sequence of pointed trees with dense orbits in $\overline{cv_N}$ that converges (non projectively) to $(T',p)$, and such that for all $n\in\mathbb{N}$, the trees $T'$ and $T'_n$ do not belong to a common closed simplex of length measures in $\overline{cv_N}$ (in the sense of \cite[Section 5]{Gui00}). Let $T\in\widehat{CV_{N+1}}$ (resp. $T_n\in\widehat{CV_{N+1}}$) be the tree associated to the graph of actions having

\begin{itemize}
\item two vertices $v_1$ and $v_2$, where the vertex tree $T_{v_1}$ is equal to $T$ (resp. $T'_n$), with attaching point $p$ (resp. $p_n$) and vertex group generated by $x_1,\dots,x_N$, and $T_{v_2}$ is reduced to a point, and $G_{v_2}$ is the cyclic subgroup of $F_{N+1}$ generated by $x_{N+1}$, and

\item a single edge of length $1$ joining $v_1$ and $v_2$, with trivial stabilizer.
\end{itemize}

\noindent The sequence $(T_n)_{n\in\mathbb{N}}$ converges in the axes topology to $T$ by Guirardel's Reduction Lemma \cite[Section 4]{Gui98}. However, for all $n\in\mathbb{N}$, we have $\widehat{d}(T^n,T)=+\infty$ by \cite[Proposition 5.7]{Hor14-1}. 
\end{ex}

\begin{rk}
However, Algom-Kfir has shown that the axes topology is strictly finer than the topology defined by $\widehat{d}$ in restriction to the simplicial part of $\widehat{CV_N}$ \cite[Theorem 5.12]{AK12}.
\end{rk}

\begin{ex}\emph{The primitive axes topology does not dominate the axes topology.}\\
Let $T\in\widehat{CV_N}$ be the Bass--Serre tree of an HNN-extension of the form $F_N=F_{N-1}\ast$. Let $g\in F_{N-1}$ be an element that does not belong to any proper free factor of $F_{N-1}$. Let $T'\in\overline{CV_N}$ be the tree obtained from $T$ by equivariantly folding an edge $e\subseteq T$ along $ge$. We have shown in \cite{Hor14-1} that the trees $T$ and $T'$ have the same translation length functions in restriction to $\mathcal{P}_N$. This implies that any sequence of trees $(T_n)_{n\in\mathbb{N}}$ that converges to $T'$ in $\overline{CV_N}$ for the axes topology, does not converge in $\widehat{CV_N}$ for the axes topology. However, such a sequence converges to $T\in\widehat{CV_N}$ for the primitive axes topology.  
\end{ex}

\subsection{Folding paths and geodesics}

Let $T\in CV_N$, and $T'\in\overline{CV_N}$ be a tree with dense orbits. A \emph{$\widehat{d}$-geodesic ray} from $T$ to $T'$ is a path $\gamma:\mathbb{R}_+\to\widehat{CV_N}$ such that for all $s\le t\in\mathbb{R}_+$, we have $$\widehat{d}(\gamma(s),\gamma(t))=t-s,$$ and the trees $\gamma(t)$ converge to $T'$ for the axes topology on $\overline{CV_N}$ as $t$ goes to $+\infty$. Using the classical construction of folding paths (see \cite{FM11,FM13,GL07,Mei13}), one shows the following fact. We sketch a proof for completeness.

\begin{prop}\label{geodesic-completion}
For all $T\in CV_N$ and all $T'\in\overline{CV_N}$ having dense orbits, there exists a $\widehat{d}$-geodesic ray in $\widehat{CV_N}$ from $T$ to $T'$. 
\end{prop}

\begin{proof}
Let $f:T\to T'$ be an optimal map, and $g\in F_N$ be a legal element for $f$ in $T$, whose axis in $T$ is contained in the tension graph of $f$, i.e. the subgraph made of those edges in $T$ that are maximally stretched by $f$ (the reader is referred to \cite[Section 6.2]{Hor14-1} for definitions and a proof of the existence of such an element $g\in F_N$). We fix representatives of $T$ and $T'$ in $\overline{cv_N}$, again denoted by $T$ and $T'$, slightly abusing notations. 

We define a simplicial tree $\overline{T}\in\overline{cv_N}$ that belongs to the same closed simplex as $T$, in the following way. We first collapse all edges in $T$ which are mapped to a point by $f$. We then shrink all edges outside of the tension graph of $f$, so that all edges in $\overline{T}$ are stretched by a factor of $M:=\text{Lip}(T,T')$ under the map $\overline{f}:\overline{T}\to T'$ induced by $f$. Denote by $K$ the distance (for $\widehat{d}$) from the covolume $1$ representative of $T$ to the covolume $1$ representative of $\overline{T}$ in $\widehat{CV_N}$. Let $(\gamma_1(t))_{t\in [0,K]}$ be a straight segment of length $K$ (staying in a closed simplex of $\widehat{CV_N}$) joining $T$ to $\widehat{T}$, parameterized by arc length. 

There exists a morphism $f:M\overline{T}\to T'$. Let $(T_t)_{t\in\mathbb{R}_+}$ be the folding path guided by $f$ constructed in \cite[Section 3]{GL07}. Notice that for all $t\in\mathbb{R}_+$, the tree $T_t$ has trivial arc stabilizers, because $T$ has trivial arc stabilizers. If the tree $T_t$ had dense orbits for some $t\in\mathbb{R}_+$, then we would have $T_t=T'$, since no folding can occur starting from a tree with dense orbits \cite[Corollary 3.10]{Hor14-1}. Denoting by $t_0$ the smallest such $t\in\mathbb{R}_+$, the sequence $(T_{t_0-\frac{1}{n}})_{n\in\mathbb{N}}$ would then be a Cauchy sequence converging to $T'$, contradicting Theorem \ref{AK}. For all $t\in\mathbb{R}_+$, we denote by $\gamma_2(t)$ the projection of $T_t$ to $\widehat{CV_N}$. 

Let $\gamma$ be the path in $\widehat{CV_N}$ defined as the concatenation of the paths $\gamma_1$ and $\gamma_2$. As the axis of $g$ is contained in the tension graph of $T$, it does not get shortened when passing from $T$ to $\overline{T}$ (and lengths do not increase when passing from $T$ to $\overline{T}$). Legality of $g$ implies that its axis never gets folded along the path $\gamma_2$. Therefore, for all $s\le t\in\mathbb{R}_+$, we have $$\widehat{d}(\gamma(s),\gamma(t))=\log\frac{||g||_{\gamma(t)}}{||g||_{\gamma(s)}}.$$ This shows that for all $s\le t \le u\in\mathbb{R}_+$, we have $\widehat{d}(\gamma(s),\gamma(u))=\widehat{d}(\gamma(s),\gamma(t))+\widehat{d}(\gamma(t),\gamma(u))$. Therefore, up to reparameterization, the path $\gamma$ is a $\widehat{d}$-geodesic ray that converges to $T'$.
\end{proof}

\subsection{Busemann points}

Let $X$ be a (possibly asymmetric) quasi-proper metric space. A path $\gamma:\mathbb{R}_+\to X$ is an \emph{almost geodesic ray} if for all $\epsilon>0$, there exists $t_0\in\mathbb{R}_+$ such that for all $s,t\ge t_0$, we have $|d(\gamma(0),\gamma(s))+d(\gamma(s),\gamma(t))-t|\le\epsilon$. Rieffel proved that every almost geodesic ray converges to a point in $X(\infty)$ \cite[Theorem 4.7]{Rie02}. A horofunction is called a \emph{Busemann point} if there exists an almost geodesic converging to it. We denote by $X_B(\infty)$ the subspace of $X(\infty)$ consisting of Busemann points.

Walsh showed that in the case of the Teichmüller space of a surface, equipped with Thurston's asymmetric metric, all horofunctions are Busemann points, since they are limits of \emph{stretch lines}, which are geodesics for Thurston's  metric, see \cite[Theorem 4.1]{Wal11}. We prove the following characterization of Busemann points in the horoboundary of outer space. Given a tree $T\in\overline{CV_N}^{prim}$, we denote by $\psi_T$ the corresponding horofunction.

\begin{theo} \label{Busemann-dense}
For all $T\in\overline{CV_N}^{prim}$, the following assertions are equivalent.
\begin{itemize}
\item The tree $T$ has dense orbits.
\item The horofunction $\psi_T$ is a Busemann point.
\item The horofunction $\psi_T$ is the limit of a $\widehat{d}$-geodesic ray in $\widehat{CV_N}$.
\item The horofunction $\psi_T$ is unbounded from below.
\end{itemize}
\end{theo}

\begin{proof}
It follows from \cite[Lemma 5.2]{Wal11} that horofunctions corresponding to Busemann points are unbounded from below.
\\
\\
\indent Let $b\in CV_N$, and let $T\in\overline{CV_N}^{prim}$ be a tree with dense orbits. Theorem \ref{geodesic-completion} gives the existence of a $\widehat{d}$-geodesic $\gamma$ starting at $b$ and converging to $T$ in $\widehat{CV_N}$. By slightly perturbing $\gamma$, we will construct an almost geodesic ray staying in $CV_N$ and converging to $T$.

We define by induction a sequence $(\gamma'(n))_{n\in\mathbb{N}}\in CV_N^{\mathbb{N}}$ satisfying $\widehat{d}(\gamma'(n),\gamma(n))\le\frac{1}{n}$ for all $n\in\mathbb{N}$, and $n-k-\frac{2}{k}\le d(\gamma'(k),\gamma'(n))\le n-k+\frac{2}{k}$ for all $k<n$, in the following way. We let $\gamma'(0):=\gamma(0)$. Let now $n\in\mathbb{N}$, and assume that $\gamma'(k)$ has already been defined for all $k<n$. Since $\widehat{d}(\gamma'(k),\gamma(k))\le\frac{1}{k}$ for all $k<n$, and as $\gamma$ is a $\widehat{d}$-geodesic ray in $\widehat{CV_N}$, by the triangle inequality, we have $\widehat{d}(\gamma'(k),\gamma(n))\le n-k+\frac{1}{k}$. By definition of $\widehat{d}$, we can choose $\gamma'(n)\in CV_N$ so that 

\begin{itemize}
\item we have $\widehat{d}(\gamma'(n),\gamma(n))\le \frac{1}{n}$, and 
\item for all $k<n$, we have $d(\gamma'(k),\gamma'(n))\le n-k+\frac{2}{k}$, and
\item we have $n-\frac{1}{n}\le d(\gamma'(0),\gamma'(n))\le n+\frac{1}{n}$.
\end{itemize}

\noindent The triangle inequality then ensures that for all $k\le n$, we have 
\begin{displaymath}
\begin{array}{rl}
d(\gamma'(k),\gamma'(n))&\ge d(\gamma'(0),\gamma'(n))-d(\gamma'(0),\gamma'(k))\\
&\ge n-k-\frac{2}{k}.
\end{array}
\end{displaymath}

 We then extend $\gamma'$ to a piecewise-geodesic ray $\gamma':\mathbb{R}_+\to CV_N$ by adding a geodesic segment joining $\gamma'(n)$ to $\gamma'(n+1)$ for all $n\in\mathbb{N}$. Let $t_0\in\mathbb{R}$ be such that $\frac{7}{\lfloor t_0\rfloor}\le\epsilon$. For all $t_0\le s\le t$, letting $n:=\lfloor s\rfloor$ and $m:=\lfloor t\rfloor$, the sum $d(\gamma'(0),\gamma'(s))+d(\gamma'(s),\gamma'(t))$ is bounded above by $$d(\gamma'(0),\gamma'(n))+d(\gamma'(n),\gamma'(n+1))+d(\gamma'(n+1),\gamma'(m))+d(\gamma'(m),\gamma'(t))\le t+\epsilon,$$ and on the other hand we have
\begin{displaymath}
\begin{array}{rl}
d(\gamma'(0),\gamma'(s))+d(\gamma'(s),\gamma'(t))&\ge d(\gamma'(0),\gamma'(t))\\
&\ge d(\gamma'(0),\gamma'(m+1))-d(\gamma'(t),\gamma'(m+1))\\
&\ge t-\epsilon. 
\end{array}
\end{displaymath}

\noindent Hence $\gamma'$ is an almost geodesic ray. In particular it converges to some $\xi\in \overline{CV_N}^{prim}$, and $\xi=T$ by construction. Hence $T$ is a Busemann point. 
\\
\\
\indent If $T$ does not have dense orbits, then we can choose a representative $\widetilde{T}\in \overline{cv_N}$ of quotient volume $1$. As $T$ is minimal, for all $x\in CV_N$ (which we identify with its covolume $1$ representative), any $F_N$-equivariant map from $x$ to $\widetilde{T}$ has Lipschitz constant at least $1$. Hence for all $x\in CV_N$, we have $$\xi_T(x)\ge\log\frac{1}{\text{Lip}(b,\widetilde{T})},$$ so $\xi_T$ is bounded below.
\end{proof}

Hence the horoboundary of outer space is naturally partitioned into three subsets, namely

\begin{itemize}
\item trees having dense orbits, which coincide with the set of Busemann points, i.e. those points that are limits of almost geodesic rays (or of geodesic rays in the completion $\widehat{CV_N}$), and
\item trees without dense orbits and with trivial arc stabilizers, which coincide with completion points, i.e. those points that are limits of Cauchy sequences, and
\item trees having nontrivial arc stabilizers.
\end{itemize}

\section{Geodesic currents and the backward horoboundary of outer space}\label{sec-backward}

As $d$ is not symmetric, we can also consider the horocompactification of outer space for the metric $d^{back}$ defined by $d^{back}(X,Y)=d(Y,X)$ for all $X,Y\in CV_N$, which satisfies the hypotheses of Proposition \ref{horocompactification-general} as $d$ does. We denote by $\overline{CV_N}^{back}$ this compactification of outer space. In this section, we investigate some properties of $\overline{CV_N}^{back}$, which show that $\overline{CV_N}^{prim}$ and $\overline{CV_N}^{back}$ are rather different in nature. It seems that there is some kind of duality between the two compactifications. Having a more explicit description of this duality and of the backward horocompactification would be of interest. For example, is the backward horocompactification isomorphic to Reiner Martin's compactification of outer space \cite[Section 6.3]{Mar95} ? The same question is also still open in the context of Teichmüller spaces equipped with Thurston's asymmetric metric. We start by recalling the notion of geodesic currents on $F_N$.

\subsection{Geodesic currents}

Let $\partial^2F_N:=\partial F_N\times\partial F_N\smallsetminus\Delta$, where $\Delta$ is the diagonal, and denote by $i:\partial^2 F_N\to\partial^2F_N$ the involution that exchanges the factors. A \emph{current} on $F_N$ is an $F_N$-invariant Borel measure $\nu$ on $\partial^2F_N$ that is finite on compact subsets of $\partial^2F_N$, see \cite{Kap05,Kap06}. We denote by $Curr_N$ the space of currents on $F_N$, equipped with the weak-$\ast$ topology, and by $\mathbb{P}Curr_N$ the space of projective classes (i.e. homothety classes) of currents.

To every $g\in F_N$ which is not of the form $h^k$ for any $h\in F_N$ and $k>1$ (we say that $g$ is not a \emph{proper power}), one associates a \emph{rational current} $[g]$ by letting $[g](S)$ be the number of translates of $(g^{-\infty},g^{+\infty})$ that belong to $S$ (where $g^{-\infty}:=\lim_{n\to +\infty} g^{-n}$ and $g^{+\infty}:=\lim_{n\to +\infty} g^n$) for all clopen subsets $S\subseteq\partial^2 F_N$, see \cite[Definition 5.1]{Kap06}. For the case of proper powers, one may set $[h^k]:=k[h]$. The group $\text{Out}(F_N)$ acts on $Curr_N$ on the left in the following way \cite[Proposition 2.15]{Kap06}. Given a compact set $K\subseteq\partial^2F_N$, an element $\Phi\in\text{Out}(F_N)$, and a current $\nu\in Curr_N$, we set $\Phi(\nu)(K):=\nu(\phi^{-1}(K))$, where $\phi\in\text{Aut}(F_N)$ is any representative of $\Phi$. The action of $\text{Out}(F_N)$ on $\mathbb{P}Curr_N$ is not minimal, but there is a unique closed (hence compact), minimal, $\text{Out}(F_N)$-invariant subset $\mathbb{P}M_N\subseteq\mathbb{P}Curr_N$, which is the closure of rational currents associated to primitive conjugacy classes of $F_N$, see \cite[Theorem B]{KL09}. We denote by $M_N$ the lift of $\mathbb{P}M_N$ to $Curr_N$. In \cite[Section 5]{Kap06}, Kapovich defined an intersection form between elements of $cv_N$ and currents, which was then extended by Kapovich and Lustig to trees in $\overline{cv_N}$ \cite{KL09}.

\begin{theo} \label{intersection-form} (Kapovich--Lustig \cite[Theorem A]{KL09})
There exists a unique $\text{Out}(F_N)$-invariant continuous function $$\langle .,.\rangle : \overline{cv_N}\times Curr_N\to\mathbb{R}_+$$ which is $\mathbb{R}_+$-homogeneous in the first coordinate and $\mathbb{R}_+$-linear in the second, and such that for all $T\in\overline{cv_N}$, and all $g\in F_N\smallsetminus\{e\}$, we have $\langle T,[g]\rangle = ||g||_T$.
\end{theo}

Two currents $\mu,\mu'\in\text{Curr}_N$ are \emph{translation-equivalent} if for all $T\in cv_N$, we have $\langle T,\mu\rangle=\langle T,\mu'\rangle$. This descends to an equivalence relation $\sim$ on $\mathbb{P}\text{Curr}_N$ by setting $[\mu]\sim [\nu]$ if there exist representatives $\mu$ and $\nu$ which are translation-equivalent.

Let $A$ be a free basis of $F_N$. Let $w$ be a nontrivial cyclically reduced word written in the basis $A$. The \emph{Whitehead graph} of $w$ in the basis $A$ is the labelled undirected graph $Wh_A(w)$ defined as follows. The vertex set of $Wh_A(w)$ is $A^{\pm 1}$. For all $x\neq y\in A^{\pm 1}$, there is an edge between $x$ and $y$ labelled by the number of occurrences of the word $xy^{-1}$ in the cyclic word $w$, i.e. the number of those $i\in\{0,\dots,|w|-1\}$ such that the infinite word $www\dots$ begins with $w_ixy^{-1}$, where $w_i$ is the initial segment of $w$ of length $i$. The \emph{Whitehead graph} of a nontrivial conjugacy class $[g]$ of elements of $F_N$ in the basis $A$ is $Wh_A([g]):=Wh_A(w)$, where $w$ is any cyclically reduced word representing $[g]$ in $A$. More generally, given a linear combination $\eta=\lambda_1[g_1]+\dots+\lambda_k[g_k]$ of rational currents, the \emph{Whitehead graph} $Wh_A(\eta)$ is the labelled undirected graph defined as follows. The vertex set of $Wh_A(\eta)$ is $A^{\pm 1}$. For all $x\neq y\in A^{\pm 1}$, there is an edge between $x$ and $y$ labelled by $\lambda_1\alpha_1+\dots+\lambda_k\alpha_k$, where $\alpha_i$ is the label of the edge between $x$ and $y$ in $Wh_A([g_i])$. The following proposition was proven in \cite{KLSS07} in the case of conjugacy classes of $F_N$ (i.e. rational currents), however its proof still works in the case of linear combinations of rational currents. 

\begin{prop}\label{translation-equivalence} (Kapovich--Levitt--Schupp--Shpilrain \cite[Theorem A]{KLSS07})
Let $\eta,\eta'\in \text{Curr}_N$ be two linear combinations of rational currents. Then $\eta$ and $\eta'$ are translation-equivalent if and only if for all free bases $A$ of $F_N$, we have $Wh_A(\eta)=Wh_A(\eta')$. 
\end{prop}

\begin{prop}\label{non-equivalent-currents}
Let $x_1,x_2\in F_N$ be two elements that belong to a common free basis. Let $\eta,\eta'\in \text{Span}\{[x_1x_2^i]\}_{i\in\mathbb{N}}$. If $\eta$ and $\eta'$ are translation-equivalent, then $\eta=\eta'$.
\end{prop}

\begin{proof}
Let $\{x_1,x_2,\dots,x_N\}$ be a free basis of $F_N$ that contains $x_1$ and $x_2$. There exists $k\in\mathbb{N}$ and real numbers $\lambda_1,\dots,\lambda_k,\lambda'_1,\dots,\lambda'_k$ such that $\eta=\sum_{i=1}^k\lambda_i[x_1x_2^i]$ and $\eta'=\sum_{i=1}^k\lambda'_i[x_1x_2^i]$. Assume that $\eta\neq\eta'$, and let $i\in\{1,\dots,k\}$ be such that $\lambda_i\neq\lambda'_i$. The set $B:=\{x_1x_2^i,x_2,\dots,x_N\}$ is a free basis of $F_N$. The edge joining $x_1x_2^i$ to $(x_1x_2^i)^{-1}$ has label $\lambda_i$ in $Wh_B(\eta)$ and $\lambda'_i$ in $Wh_B(\eta')$, so Proposition \ref{translation-equivalence} implies $\eta$ and $\eta'$ are not translation-equivalent, a contradiction. Hence $\eta=\eta'$.
\end{proof}

We notice the following property of the currents we considered in Proposition \ref{non-equivalent-currents}.

\begin{prop}\label{span-minimal}
Let $N\ge 3$, and let $x_1,x_2\in F_N$ be two elements that belong to a common free basis of $F_N$. Then $\text{Span}\{[x_1x_2^i]\}_{i\in\mathbb{N}}\subseteq M_N$. 
\end{prop}

\begin{proof}
Let $F<F_N$ be the free factor generated by $x_1$ and $x_2$. By \cite[Proposition 12.1]{Kap06}, there is a linear topological embedding $\iota:Curr(F)\to Curr_N$ such that for all $g\in F$, we have $\iota([g])=[g]$. In particular, the subspace $\text{Span}_{Curr_N}\{[x_1x_2^i]\}_{i\in\mathbb{N}}$ identifies with the image $\iota(\text{Span}_{Curr(F)}\{[x_1x_2^i]\}_{i\in\mathbb{N}})$. The proof of \cite[Proposition 4.3]{KL07} thus shows that $\text{Span}_{Curr_N}\{[x_1x_2^i]\}_{i\in\mathbb{N}}\subseteq M_N$.
\end{proof}

\subsection{The backward horoboundary of outer space}

We recall that $b$ denotes some fixed basepoint in $CV_N$. For all $z\in CV_N$, we define the function 
\begin{displaymath}
\begin{array}{cccc}
\psi_z^{back}:&CV_N&\to&\mathbb{R}\\
&x&\mapsto &d(z,x)-d(z,b). 
\end{array}
\end{displaymath}

\noindent Given a finite set $S\subseteq M_N$, we define a function $f_S$ on $CV_N$ by setting $$f_{S}(T):=\log\frac{\sup_S\langle T,\mu\rangle}{\sup_S\langle b,\mu\rangle}$$ for all $T\in CV_N$.

\begin{prop} \label{backward-horoboundary}
For all $\xi\in\overline{CV_N}^{back}$, there exists a finite set $S\subseteq M_N$ such that $\xi=f_S$.
\end{prop}

\begin{proof}
Let $\xi\in\overline{CV_N}^{back}$, and let $(T_n)_{n\in\mathbb{N}}\in CV_N^{\mathbb{N}}$ be a sequence of elements of $CV_N$ that converges to $\xi$. For all $n\in\mathbb{N}$, let $\mathcal{F}_n$ denote the set of candidates in $T_n$. There is a uniform bound on the cardinality of $\mathcal{F}_n$. Up to passing to a subsequence, we can thus assume that there exists $k\in\mathbb{N}$, currents $\eta^1,\dots,\eta^k\in M_N$, and  sequences $(g_n^i)_{n\in\mathbb{N}}\in\prod_{n\in\mathbb{N}}\mathcal{F}_n$ and $(\lambda_n^i)_{n\in\mathbb{N}}\in\mathbb{R}^{\mathbb{N}}$ for all $i\in\{1,\dots,k\}$, such that the sequence $(\lambda_n^i [g_n^i])_{n\in\mathbb{N}}$ converges (non-projectively) to the current $\eta^i$, and for all $n\in\mathbb{N}$, we have $\mathcal{F}_n=\{g_n^i\}_{1\le i\le k}$. For all $n\in\mathbb{N}$ and all $i,j\in\{1,\dots,k\}$, we set $$\alpha_n^{i,j}:=\frac{\lambda_n^i\langle T_n,[g_n^i]\rangle}{\lambda_n^j\langle T_n,[g_n^j]\rangle}.$$ Up to passing to a subsequence again, we may assume that for all $i,j\in\{1,\dots,k\}$, the sequence $(\alpha_n^{i,j})_{n\in\mathbb{N}}$ converges in $\mathbb{R}\cup \{+\infty\}$. Denoting its limit by $\alpha^{i,j}$, we can find $i_0\in\{1,\dots,k\}$ such that for all $j\in\{1,\dots,k\}$, we have $\alpha^{i_0,j}<+\infty$. We set $$S:=\{\alpha^{i_0,j}\eta^j|\alpha^{i_0,j}\neq 0\}.$$ For all $T\in CV_N$ and all $n\in\mathbb{N}$, we have 
\begin{displaymath}
\begin{array}{rl}
\psi_{T_n}^{back}(T)&=\log (\sup_j\frac{\langle T,[g_n^j]\rangle}{\langle T_n,[g_n^j]\rangle})-\log (\sup_j\frac{\langle b,[g_n^j]\rangle}{\langle T_n,[g_n^j]\rangle})\\
&=\log (\sup_j\frac{\lambda_n^{i_0}\langle T_n,[g_n^{i_0}]\rangle}{\lambda_n^j\langle T_n,[g_n^j]\rangle}\frac{\lambda_n^j\langle T,[g_n^j]\rangle}{\lambda_n^{i_0}\langle T_n,[g_n^{i_0}]\rangle})-\log (\sup_j\frac{\lambda_n^{i_0}\langle T_n,[g_n^{i_0}]\rangle}{\lambda_n^j\langle T_n,[g_n^j]\rangle}\frac{\lambda_n^j\langle b,[g_n^j]\rangle}{\lambda_n^{i_0}\langle T_n,[g_n^{i_0}]\rangle})\\
&=\log(\sup_j\alpha_n^{i_0,j}\lambda_n^j\langle T,[g_n^j]\rangle)-\log(\sup_j\alpha_n^{i_0,j}\lambda_n^j\langle b,[g_n^j]\rangle),
\end{array}
\end{displaymath}

\noindent which tends to $f_S(T)$ as $n$ goes to $+\infty$. Hence $\xi=f_S$.
\end{proof}

\begin{rk}\label{rk-dual}
It follows from the proof of Proposition \ref{backward-horoboundary} that if a sequence $(T_n)_{n\in\mathbb{N}}\in CV_N^{\mathbb{N}}$ converges to a horofunction $f_S$ in the backward horoboundary of $CV_N$, then all currents in $S$ are dual to all limit points of $(T_n)_{n\in\mathbb{N}}$ in $\overline{CV_N}$.
\end{rk}

\begin{prop}\label{minimal-in-back}
There exists a topological embedding from $\mathbb{P}M_N/{\sim}$ to $\overline{CV_N}^{back}$.
\end{prop}

\begin{proof}
Let $\eta\in M_N$, and let $(g_n)_{n\in\mathbb{N}}\in\mathcal{P}_N^{\mathbb{N}}$ be a sequence of primitive elements so that the rational currents $[g_n]$ projectively converge to $\eta$. For all $n\in\mathbb{N}$, let $T_n\in CV_N$ be a rose, one of whose petals is labelled by $g_n$. By making the length of this petal arbitrarily small, we can ensure, with the notations from the proof of Proposition \ref{backward-horoboundary} (where we assume that $g_n^1:=g_n$), that $\alpha^{1,j}=0$ for all $j>1$. This implies that the functions $\psi_{T_n}^{back}$ converge pointwise (and hence uniformly on compact sets of $d_{sym}$) to $f_{\{\eta\}}$. Therefore, we get an injective map from the compact space $\mathbb{P}M_N/\sim$ to the Hausdorff space $\overline{CV_N}^{back}$, which is continuous by continuity of the intersection form (Theorem \ref{intersection-form}).
\end{proof}

\begin{cor}
For all $N\ge 3$, the space $\overline{CV_N}^{back}$ has infinite topological dimension.
\end{cor}

\begin{proof}
Let $\{x_1,\dots,x_N\}$ be a free basis of $F_N$. Propositions \ref{non-equivalent-currents}, \ref{span-minimal} and \ref{minimal-in-back} show that $\overline{CV_N}^{back}$ contains an embedded copy of the infinite-dimensional projective space spanned by all currents of the form $[x_1x_2^i]$ for $i\in\mathbb{N}$. The claim follows.
\end{proof}

\subsection{The backward horocompactification of $CV_2$}

We finish this section by giving a description of the backward horocompactification of $CV_2$. Culler and Vogtmann gave in \cite{CV91} an explicit description of $\overline{CV_2}$, and an explicit description of the primitive compactification $\overline{CV_2}^{prim}$ was given in \cite[Section 2.2]{Hor14-1}. We will show that the backward horocompactification of $CV_2$ is $2$-dimensional, homeomorphic to a disk with fins attached, see Figure \ref{fig-back}. The reduced part of this compactification, obtained by collapsing the fins, is isomorphic to the reduced part of the forward horocompactification. However, when we include the fins, there are examples of sequences of trees that converge to a point in the forward horoboundary, but not in the backward horoboundary, and vice versa. 

Every outer automorphism of $F_2$ can be realized by a mapping class of a torus with one boundary component. The set of rational geodesic currents associated to essential simple closed curves on the surface that are not isotopic to the boundary curve is $\text{Out}(F_2)$-invariant, and its closure consists of currents associated to measured laminations on the surface. Every such lamination is either minimal and filling, or a simple closed curve on the surface. The minimal set of currents $M_2$ consists of currents of this form. Currents $\eta$ associated to filling laminations are dual to a unique tree $T$, and by unique ergodicity the current $\eta$ is then the unique current in $M_2$ dual to $T$, up to homothety. This implies in particular that currents in $M_2$ that are dual to simplicial trees are rational currents corresponding to primitive elements of $F_2$.

Let $(T_n)_{n\in\mathbb{N}}\in CV_2^{\mathbb{N}}$ be a sequence that converges to a horofunction $f_S$ in $\overline{CV_2}^{back}$. Up to passing to a subsequence, we can assume that $(T_n)_{n\in\mathbb{N}}$ also converges to a tree $T\in\overline{CV_2}$. By \cite{CV91}, the tree $T$ is either simplicial, or dual to an arational measured foliation on a torus with one boundary component. The list of simplicial trees in $\partial{CV_2}$ is displayed on Figure \ref{fig-list}. 

\begin{figure}
\begin{center}
\input{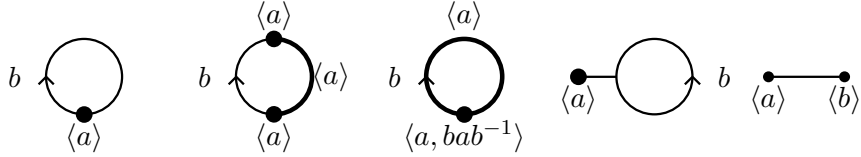}
\caption{The simplicial trees in $\partial CV_2$.}
\label{fig-list}
\end{center}
\end{figure}

First assume that $T$ is a simplicial metric tree, whose quotient graph $T/F_2$ has one of the first four shapes displayed on Figure \ref{fig-list}. It then follows from the above description of $M_2$ that there is a unique projectivized current $\eta\in M_2$ that is dual to $T$ (which is a rational current, associated to $a$). This is clear except in the case where the quotient graph $T/F_2$ has the third shape displayed on Figure \ref{fig-list}. In this case, since $T$ is simplicial, all currents in $M_2$ dual to $T$ are rational. In addition, it follows from \cite[Proposition 2.8]{Hor14-1} that the only primitive elements of $F_2$ contained in the subgroup $\langle a,bab^{-1}\rangle$ are conjugate to $a$. Remark \ref{rk-dual} implies that $S=\{[a]\}$.

Notice in particular that all sequences of trees that converge to a simplicial metric tree $T$ of the fourth type in $\overline{CV_2}$ converge to the same horofunction $f_{\{[a]\}}$ in $\overline{CV_2}^{back}$, regardless of the ratio between the lengths of the separating edge and the nonseparating edge in the quotient graph $T/F_2$. Let $(T_n)_{n\in\mathbb{N}}$ be a sequence of Bass--Serre trees of barbell graphs whose petals are labelled by $a$ and $b$, where the length of the petal labelled by $a$ (respectively by $b$) converges to $0$ (resp. to $l\in(0,1]$) and whose separating edge has a length converging to $1-l$. Then $(T_n)_{n\in\mathbb{N}}$ converges to $f_{\{[a]\}}$ in $\overline{CV_2}^{back}$, and it converges in $\overline{CV_2}^{prim}$ to the Bass--Serre tree of a graph having the fourth shape displayed on Figure \ref{fig-list}, whose separating edge has length $1-l$, and whose nonseparating edge has length $l$. By making $l$ vary, we get examples of sequences of trees that converge to the same point in $\overline{CV_2}^{back}$, but to distinct trees in $\overline{CV_2}^{prim}$.

Also notice that all sequences of trees that converge in $\overline{CV_2}$ to a simplicial tree $T$ having one of the first three shapes displayed, converge to the same horofunction $f_{\{[a]\}}$ in $\overline{CV_2}^{back}$. All these trees are also identified by the quotient map $\overline{CV_2}\to\overline{CV_2}^{prim}$. 

Now assume that $T$ is dual to an arational measured foliation on a torus with a single boundary component. As noticed above, there is a unique projectivized current $\eta\in \mathbb{P}M_2$ that is dual to $T$. Remark \ref{rk-dual} implies that $S=\{\eta\}$.  

In the remaining case where $T$ is the Bass--Serre tree of a splitting of the form $F_2=\langle a\rangle\ast\langle b\rangle$, there are exactly $2$ projectivized currents in $\mathbb{P}M_2$ that are dual to $T$: these are the currents whose lifts to $M_2$ are $[a]$ and $[b]$. We claim that the set $S$ may consist of any pair of the form $\{\lambda_1[a],\lambda_2[b]\}$, where we may assume that $\lambda_1+\lambda_2=1$ because multiplying all currents by a same factor does not change the map $f_S$. Indeed, first assume that $\lambda_1,\lambda_2>0$. For all $n\in\mathbb{N}$, we let $T_n$ be the Bass--Serre tree of a barbell graph, whose central edge has length $1$, and whose loops are labelled by $a$ and $b$ and have respective lengths $\frac{1}{\lambda_1 n}$ and $\frac{1}{\lambda_2 n}$. Then $(T_n)_{n\in\mathbb{N}}$ converges to $f_{\{\lambda_1[a],\lambda_2[b]\}}$ in $\overline{CV_2}^{back}$. If $\lambda_1=0$ and $\lambda_2=1$, then we let $T_n$ be a barbell whose loops are labelled by $a$ and $b$ and have respective lengths $1$ and $\frac{1}{n}$ to get the desired convergence in $\overline{CV_2}^{back}$. In all cases, the sequence $(T_n)_{n\in\mathbb{N}}$ converges in $\overline{CV_2}^{prim}$ to the Bass--Serre tree of the splitting $F_2=\langle a\rangle\ast\langle b\rangle$. This provides examples of sequences of trees that converge to the same point in $\overline{CV_2}^{prim}$, but to distinct points in $\overline{CV_2}^{back}$. The closure of the simplex of a barbell graph in $\overline{CV_2}^{back}$ is displayed on Figure \ref{fig-barbell}. 

We claim that the backward horoboundary $\overline{CV_2}^{back}$ is isomorphic to the forward horoboundary, in which the closures of the simplices of barbell graphs have been replaced by simplices having the shape displayed on Figure \ref{fig-barbell}. Indeed, there is a bijection between $\mathbb{P}M_2$ and the set of simplicial trees in $\overline{CV_2}^{prim}$ that do not contain any separating edge. It thus follows from the above that a sequence $(T_n)_{n\in\mathbb{N}}\in CV_2^{\mathbb{N}}$ converges to a horofunction $f_{\{\eta\}}$ with $\eta\in M_2$ if and only if all its limit points in $\overline{CV_2}^{prim}$ are dual to $\eta$. From this observation, one deduces that the \emph{reduced} parts of $\overline{CV_2}^{back}$ and $\overline{CV_2}^{prim}$ (obtained by forgetting trees in $CV_2$ whose quotient graphs contain a separating edge) coincide. A sequence in $\overline{CV_2}^{back}$ can converge to a horofunction of the form $f_{\{\lambda_1 [a],\lambda_2 [b]\}}$ only if it eventually stays in the corresponding simplex. We note that a sequence $(T_n)_{n\in\mathbb{N}}\in CV_2^{\mathbb{N}}$ of barbell graphs with petals labelled by $a$ and $b$ converges to $f_{\{\lambda_1 [a],\lambda_2 [b]\}}$ if and only if the ratio $\frac{||b||_{T_n}}{||a||_{T_n}}$ converges to $\frac{\lambda_1}{\lambda_2}$. One also checks that all horofunctions described above are pairwise distinct. 

\begin{figure}
\begin{center}
\input{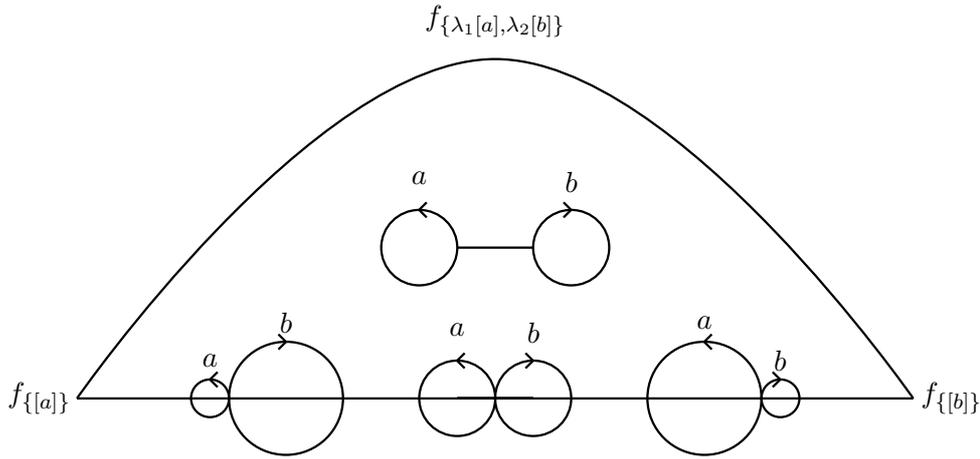}
\caption{The closure of the simplex of a barbell graph in $\overline{CV_2}^{back}$.}
\label{fig-barbell}
\end{center}
\end{figure}

\begin{figure}
\begin{center}
\def\JPicScale{.8}
\input{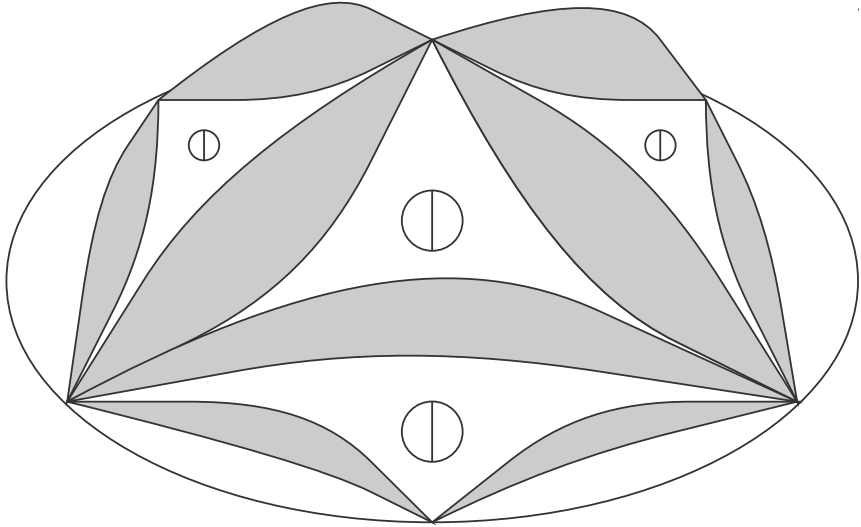}
\caption{The backward horocompactification of $CV_2$.}
\label{fig-back}
\end{center}
\end{figure}

\section{Growth of elements of $F_N$ under random products of automorphisms}

In this section, we will use our description of the horoboundary of outer space to derive results about random products of outer automorphisms of a finitely generated free group, through the study of the possible growth rates of elements of $F_N$ under such products. This is inspired from Karlsson's analogous work for random products of mapping classes of a surface \cite{Kar12}. 

\subsection{Background on ergodic cocycles}

Let $(\Omega,\mathcal{A},\mathbb{P})$ be a standard probability Lebesgue space, and $T:\Omega\to\Omega$ an ergodic measure-preserving transformation. Let $\phi:\Omega\to\text{Out}(F_N)$ be a measurable map. We call $$\Phi_n(\omega):=\phi(T^{n-1}\omega)\dots \phi(\omega)$$ an \emph{ergodic cocycle}. We say that it is \emph{integrable} if $$\int_{\Omega} d_{CV_N}^{sym}(\phi(\omega)b,b)<+\infty,$$ where we recall that $b$ is any basepoint in $CV_N$. The particular case where $\Omega=(\text{Out}(F_N)^{\mathbb{N}},\mu^{\otimes\mathbb{N}})$ is a product probability space (here $\mu$ denotes a probability measure on $\text{Out}(F_N)$), and $T$ is the shift operator, corresponds to the \emph{left random walk} on $(\text{Out}(F_N),\mu)$. This is the Markov chain on $\text{Out}(F_N)$ whose initial distribution is the Dirac measure at the identity, and with transition probabilities $p(x,y):=\mu(yx^{-1})$. In other words, the position of the random walk at time $n$ is given from its initial position $\Phi_0=\text{id}$ by successive multiplications on the left of independent $\mu$-distributed increments $\phi_i$, i.e. $\Phi_n=\phi_n\dots\phi_1$.

The \emph{drift} of an integrable ergodic cocycle $\Phi_n$ is defined as $$l:=\lim_{n\to +\infty}\frac{1}{n}d(b,\Phi_n(\omega)^{-1}b),$$ which almost surely exists by Kingmann subadditive ergodic theorem \cite{Kin68}, and is independent of $\omega$ by ergodicity of $T$.

\subsection{Growth of elements of $F_N$ under cocycles of automorphisms}\label{cocycle}

Given an element $g\in F_N$, we denote by $||g||$ the length of the cyclically reduced word that represents the conjugacy class of $g$ in some fixed basis of $F_N$ (word lengths with respect to two different bases are bi-Lipschitz equivalent). We will show the following theorem.

\begin{theo}\label{ergo-1}
Let $\Phi_n=\phi_n\dots \phi_1$ be an integrable ergodic cocycle of elements of $\text{Out}(F_N)$, and let $l$ be its drift. For $\mathbb{P}$-a.e. $\omega$, there exists a (random) tree $T(\omega)\in\overline{CV_N}$ such that

\begin{itemize}
\item for all $g\in F_N$ which are hyperbolic in $T(\omega)$, we have $$\lim_{n\to +\infty}\frac{1}{n}\log{||\Phi_n(\omega)(g)||}=l ;$$
\item for all $g\in F_N$ which are elliptic in $T(\omega)$, we have  $$\limsup_{n\to +\infty}\frac{1}{n}\log{||\Phi_n(\omega)(g)||}\le l.$$
\end{itemize}
\end{theo}

As in \cite{Kar12}, Theorem \ref{ergo-1} follows from the following (more precise) quantitative version.

\begin{theo}\label{ergo-2}
Let $\Phi_n=\phi_n\dots\phi_1$ be an integrable ergodic cocycle of elements of $\text{Out}(F_N)$, and let $l$ be its drift. For $\mathbb{P}$-a.e. $\omega$, there exist a (random) constant $C(\omega)>0$ and a (random) tree $T(\omega)\in\overline{CV_N}$ such that for all $\epsilon>0$, there exists $n_0\in\mathbb{N}$ such that for all $n\ge n_0$ and all $g\in F_N$, we have $$C(\omega)||g||_{T(\omega)}(e^l-\epsilon)^n\le ||\Phi_n(\omega)(g)||\le C(\omega) ||g||(e^l+\epsilon)^n.$$
\end{theo}

Our proof of Theorems \ref{ergo-1} and \ref{ergo-2} relies on the following theorem of Karlsson and Ledrappier \cite{KL06}, which was originally stated for symmetric metric spaces. The extension to the case of an asymmetric metric is due to Karlsson \cite{Kar12}. 

\begin{theo} (Karlsson--Ledrappier \cite{KL06}) \label{KL}
Let $T$ be a measure-preserving transformation of a Lebesgue probability space $(\Omega,\mathcal{A},\mathbb{P})$, let $G$ be a locally compact group acting by isometries on a (possibly asymmetric) quasi-proper metric space $X$, let $b\in X$, and let $\phi:\Omega\to G$ be a measurable map satisfying $$\int_{\Omega}d^{sym}(\phi(\omega)b,b)d\mathbb{P}(\omega)<+\infty.$$ Let $\Phi_n$ be the associated integrable ergodic cocycle. Then, for $\mathbb{P}$-almost every $\omega$, there exists $\xi_{\omega}\in X(\infty)$ such that $$\lim_{n\to +\infty}-\frac{1}{n}\xi_{\omega}(\Phi_n(\omega)^{-1}b)=l.$$
\end{theo}

\begin{proof}[Proof of Theorems \ref{ergo-1} and \ref{ergo-2}]
We may choose as a basepoint $b\in CV_N$ a Cayley tree of $F_N$ with respect to our fixed free basis of $F_N$, so that for all $g\in F_N$, we have $||g||_b=||g||$. Let $\Phi_n=\phi_n\dots \phi_1$ be an integrable ergodic cocycle of elements of $\text{Out}(F_N)$. Theorem \ref{KL} ensures that for almost every $\omega$, there exists $\xi=\xi(\omega)$, associated to a tree $T=T(\omega)\in \overline{CV_N}$, such that $$\lim_{n\to +\infty} -\frac{1}{n}\xi(\Phi_n(\omega)^{-1}b)=l$$ (if $l=0$ then we can choose $T(\omega)\in CV_N$, while if $l>0$, then $\xi$ is unbounded, and hence $T(\omega)\in CV_N(\infty)$ is a tree with dense orbits by Theorem \ref{Busemann-dense}). Using the expression of horofunctions given by Theorem \ref{horocompactification}, for all $\epsilon>0$, there exists $n_0\in\mathbb{N}$ such that for all $n>n_0$ one has $$\log \sup_{g\in F_N\smallsetminus\{e\}}\frac{||g||_{T}}{||g||_{\Phi_n^{-1} b}}-\log\sup_{g\in F_N\smallsetminus\{e\}}\frac{||g||_{T}}{||g||_{b}}\le -(l-\epsilon)n.$$ Letting $$C(\omega)^{-1}:=\sup_{g\in F_N\smallsetminus\{e\}}\frac{||g||_T}{||g||_{b}},$$ we obtain $$\sup_{g\in F_N}\frac{||g||_{T}}{||g||_{\Phi_n^{-1} b}}\le C(\omega)^{-1} e^{-(l-\epsilon)n},$$ so for all $g\in F_N$ we have $$||\Phi_n (g)||_{b}=||g||_{\Phi_n^{-1} b}\ge C(\omega)||g||_{T}e^{(l-\epsilon)n}.$$ As $$\lim_{n\to +\infty} \frac{1}{n} d(b,\Phi_n^{-1} b)=l,$$ for all sufficiently large $n\in\mathbb{N}$ and all $g\in F_N$, we also have $$||\Phi_n(g)||_{b}\le ||g||_{b}e^{(l+\epsilon)n}.$$ This shows Theorem \ref{ergo-2}. In view of these two inequalities we have $$\lim_{n\to +\infty}\frac{1}{n}\log||\Phi_n(g)||=l$$ for all $g\in F_N$ which are hyperbolic in $T$, showing Theorem \ref{ergo-1}.
\end{proof}

\subsection{The case of a random walk on a nonelementary subgroup of $\text{Out}(F_N)$}\label{sec-nonelementary}

A subgroup of $\text{Out}(F_N)$ is \emph{nonelementary} if it does not virtually fix the conjugacy class of any finitely generated subgroup of $F_N$ of infinite index. In the case of independent increments (i.e. $\Omega$ is a product probability space, and $T$ is the shift operator), Theorem \ref{ergo-1} specifies as follows. The following corollary is analogous to a theorem of Furstenberg for random products of matrices \cite{Fur63-2}, and to Karlsson's theorem for random products of elements of the mapping class group of a surface \cite[Corollary 4]{Kar12}.

\begin{cor} \label{random}
Let $\mu$ be a probability measure on $\text{Out}(F_N)$, whose support generates a nonelementary subgroup of $\text{Out}(F_N)$. For all $g\in F_N$, and almost every sample path $(\Phi_n)_{n\in\mathbb{N}}$ of the random walk on $(\text{Out}(F_N),\mu)$, we have $$\lim_{n\to +\infty}\frac{1}{n}\log{||\Phi_n(g)||}=l,$$ where $l$ is the drift of the random walk.
\end{cor}

\begin{proof}
In view of Theorem \ref{ergo-2}, it is enough to show that the tree $T(\omega)$ associated to the horofunction provided by Karlsson and Ledrappier's theorem, can almost surely be chosen to be free. This will be a consequence of Propositions \ref{Karlsson-Ledrappier} and \ref{nonelementary}. 
\end{proof}

\begin{rk}
When the random walk on $(\text{Out}(F_N),\mu)$ has positive drift with respect to the Lipschitz distance on $\text{Out}(F_N)$, we therefore get that all elements of $F_N$ almost surely have exponential growth along the sample path of the random walk, with the same exponential growth rate. Positivity of the drift is discussed in Section \ref{sec-drift}.  
\end{rk}

Our proof of Corollary \ref{random} relies on the following refinement of Karlsson and Ledrappier's theorem in the case of independent increments. The following statement was noticed by Karlsson in \cite[Section 2]{Kar12}, and follows from the proof of \cite[Theorem 18]{KL11-2}.

\begin{prop} (Karlsson \cite[Section 2]{Kar12}) \label{Karlsson-Ledrappier}
Let $G$ be a locally compact group acting by isometries on a (possibly asymmetric) quasi-proper metric space, and let $\mu$ be a probability measure on $G$ with finite first moment with respect to $d^{sym}$. Let $E\subseteq X(\infty)$ be a measurable subset such that for all $\mu$-stationary measures $\nu$ on $X\cup X(\infty)$, we have $\nu(E)=1$. Then for $\mu$-almost every $\omega$, the horofunction $\xi_{\omega}$ from Theorem \ref{KL} may be chosen to belong to $E$.
\end{prop}

\subsubsection{Stationary measures in the horoboundary of outer space}

Let $\mu$ be a probability measure on $\text{Out}(F_N)$. We now aim at understanding some properties of $\mu$-stationary measures on $CV_N(\infty)$. Given a probability measure $\mu$ on a countable group $G$ acting on a compact space $X$, there always exists a $\mu$-stationary Borel probability measure on $X$, obtained as a weak limit of the Cesàro averages of the convolution of $\mu^{\ast n}$ and any Borel probability measure on $X$ (see \cite{Fur73} or \cite[Lemma 2.2.1]{KM96}). Compactness of $CV_N(\infty)$ thus yields the following lemma.

\begin{lemma} \label{stationary-measure}
Let $\mu$ be a probability measure on $\text{Out}(F_N)$. Then there exists a $\mu$-stationary measure on $CV_N(\infty)$.
\qed
\end{lemma}

The following statement is essentially proved in \cite[Proposition 2.4]{Hor14-4}, we sketch a proof for completeness. We recall that we have associated a \emph{canonical} representative to every class of primitive-equivalence in Section \ref{sec-prim}.

\begin{prop}\label{nonelementary}(Horbez \cite[Proposition 2.4]{Hor14-4})
Let $\mu$ be a probability measure on $\text{Out}(F_N)$, and let $\nu$ be a $\mu$-stationary measure on $CV_N(\infty)$. Then $\nu$ is concentrated on the set of trees $T\in CV_N(\infty)$ such that all conjugacy classes of point stabilizers in the canonical lift of $T$ to $\overline{CV_N}$ have finite $gr(\mu)$-orbits. In particular, if $gr(\mu)$ is nonelementary, then every $\mu$-stationary measure on $CV_N(\infty)$ is concentrated on the set of free $F_N$-actions.
\end{prop}

Our proof of Proposition \ref{nonelementary} makes use of the following classical lemma, whose proof relies on a maximum principle argument.

\begin{lemma} \label{disjoint-translations} (Ballmann \cite{Bal89}, Woess \cite[Lemma 3.4]{Woe89}, Kaimanovich--Masur \cite[Lemma 2.2.2]{KM96}, Horbez \cite[Lemma 3.3]{Hor14-3})
Let $\mu$ be a probability measure on a countable group $G$, and let $\nu$ be a $\mu$-stationary probability measure on a $G$-space $X$. Let $D$ be a countable $G$-set, and let $\Theta:X\to D$ be a measurable $G$-equivariant map. If $E\subseteq X$ is a $G$-invariant measurable subset of $X$ satisfying $\nu(E)>0$, then $\Theta(E)$ contains a finite $gr(\mu)$-orbit.
\end{lemma}

\begin{proof}[Proof of Proposition \ref{nonelementary}]
Let $D$ be the countable set of all finite collections of conjugacy classes of finitely generated subgroups of $F_N$. For all $T\in CV_N(\infty)$, we let $\Theta(T)$ be the collection of conjugacy classes of point stabilizers in the canonical lift of $T$ to $\overline{CV_N}$. This set is finite \cite{Jia91} and belongs to $D$ by \cite[Corollary III.4]{GL95}. We have $\Theta(T)\neq\emptyset$ as soon as some representative of $T$ in $\overline{CV_N}$ is not free. We now prove that $\Theta$ is measurable, which will be enough to conclude by applying Lemma \ref{disjoint-translations} to $\Theta$. We denote by $\partial CV_N$ the boundary of Culler and Morgan's compactification of $CV_N$. The projection map $\pi:\partial CV_N\to CV_N(\infty)$ is closed, so \cite[Theorem V.3]{Cas67} and \cite[Corollary III.3]{CV77} imply that there exist countably many measurable maps $f_n:CV_N(\infty)\to\partial CV_N$, so that for all $T\in CV_N(\infty)$, we have $\pi^{-1}(T)=\overline{\{f_n(T)|n\in\mathbb{N}\}}$. Given a conjugacy class $H\in D$, we have $H\in\Theta(T)$ if and only if 

\begin{itemize}
\item for all $g\in F_N$ which is conjugate into $H$, and all $n\in\mathbb{N}$, we have $||g||_{f_n(T)}=0$, and
\item for all $g\in F_N$ which is not conjugate into $H$, there exists $n\in\mathbb{N}$ such that $||g||_{f_n(T)}\neq 0$.
\end{itemize}

\noindent Measurability of $\Theta$ follows.  
\end{proof}

\subsubsection{Drift of a random walk on a nonelementary subgroup of $\text{Out}(F_N)$} \label{sec-drift}

The \emph{free factor complex} $FF_N$, introduced by Hatcher and Vogtmann in \cite{HV98}, is defined when $N\ge 3$ as the simplicial complex whose vertices are the conjugacy classes of nontrivial proper free factors of $F_N$, and higher dimensional simplices correspond to chains of inclusions of free factors. (When $N=2$, one has to modify this definition by adding an edge between any two complementary free factors to ensure that ${FF}_2$ remains connected, and ${FF}_2$ is isomorphic to the Farey graph). We equip $FF_N$ with the simplicial metric $d_{FF_N}$, and we fix a basepoint $\ast_{FF_N}\in FF_N$ for measuring the drift of the random walk. The following theorem, due to Calegari and Maher \cite[Section 5.10]{CM13}, relies on work of Maher \cite{Mah11} and on the convergence of almost every sample path of the random walk on $(\text{Out}(F_N),\mu)$ to the Gromov boundary $\partial FF_N$ (which is also established in \cite[Theorem 4.2]{Hor14-4} by other methods). 

\begin{theo}(Calegari--Maher \cite[Theorem 5.34]{CM13})\label{Maher} 
Let $\mu$ be a probability measure on $\text{Out}(F_N)$, whose support is finite and generates a nonelementary subgroup of $\text{Out}(F_N)$ which is not virtually cyclic. Then the random walk on $(\text{Out}(F_N),\mu)$ has positive drift with respect to $d_{FF_N}$.
\end{theo}

\begin{cor}\label{drift}
Let $\mu$ be a probability measure on $\text{Out}(F_N)$, whose support is finite and generates a nonelementary subgroup of $\text{Out}(F_N)$ which is not virtually cyclic. Then the random walk on $(\text{Out}(F_N),\mu)$ has positive drift with respect to $d_{CV_N}$.
\end{cor}

\begin{proof}
Corollary \ref{drift} follows from Theorem \ref{Maher} and from the following estimate relating the distances $d_{FF_N}$ and $d_{CV_N}$.
\end{proof}

\begin{prop} \label{sphere-Lipschitz}
There exist $K,L\in\mathbb{R}$ such that for all $\Phi,\Psi\in\text{Out}(F_N)$, we have $$d_{{FF}_N}(\Phi\ast_{FF_N},\Psi\ast_{FF_N})\le K d_{CV_N}(\Phi b,\Psi b)+L.$$
\end{prop}

Proposition \ref{sphere-Lipschitz} will follow from several distance estimates between various $\text{Out}(F_N)$-complexes, provided by Lemmas \ref{factor-sphere} and \ref{sphere-intersection} and Proposition \ref{Lipschitz-intersection}. We will first introduce yet another $\text{Out}(F_N)$-complex. Let $M_N:=\#^N S^1\times S^2$ be the connected sum of $N$ copies of $S^1\times S^2$, whose fundamental group is free of rank $N$. A \emph{sphere system} is a collection of disjoint, embedded $2$-spheres in $M_N$, none of which bounds a ball, and no two of which are isotopic. The \emph{sphere complex} $\mathcal{S}_N$, introduced by Hatcher in \cite{Hat95}, is the simplicial complex whose $k$-simplices are the isotopy classes of systems of $k+1$ spheres in $M_N$ (a $(k-1)$-dimensional face of a $k$-simplex $\Delta$ is obtained by removing one sphere from the sphere system corresponding to $\Delta$). We denote by $\mathcal{S}'_N$ the one-skeleton of the first barycentric subdivision of $\mathcal{S}_N$, which we equip with the simplicial metric $d_{\mathcal{S}'_N}$. Again, we fix a basepoint $\ast_{\mathcal{S}'_N}\in\mathcal{S}'_N$. There is a coarsely well-defined, coarsely equivariant map $\tau:\mathcal{S}'_N\to FF_N$, that maps a sphere system $S$ to the conjugacy class of the fundamental group of a complementary component in $M_N$ of a sphere in $S$. The map $\tau$ is Lipschitz, so we get the following estimate.

\begin{lemma} \label{factor-sphere}
There exists $C>0$ such that for all $\Phi,\Psi\in\text{Out}(F_N)$, we have $$d_{FF_N}(\Phi\ast_{FF_N},\Psi\ast_{FF_N})\le C d_{\mathcal{S}'_N}(\Phi\ast_{\mathcal{S}'_N},\Psi\ast_{\mathcal{S}'_N}).$$
\qed
\end{lemma}

Given two sphere systems $S$ and $S'$ in $M_N$, the \emph{intersection number} $i(S,S')$ is the minimal number of intersection circles between representatives of the isotopy classes of $S$ and $S'$. Assume that $S$ and $S'$ have been isotoped so as to minimize their number of intersection circles. There is a classical surgery procedure \cite{Hat95} that creates from a sphere $s$ in $S$ two spheres $s_1$ and $s_2$ that are both disjoint from $s$, and have fewer intersection circles with $S'$. Pick an innermost disk on a sphere $s'\in S'$, bounded by a circle $C$ of intersection with $s$. The circle $C$ splits $s$ into two disks $D_1$ and $D_2$. For all $i\in\{1,2\}$, the sphere $s_i$ consists of a parallel copy of $D_i$ attached to a parallel copy of $D$, see Figure \ref{fig-surgery}. Notice that all other intersection circles between $s$ and $S'$ are distributed over $s_1$ and $s_2$. In particular, there exists $j\in\{1,2\}$ such that $i(s_j,S')\le\frac{i(s,S')}{2}$ (isotoping $s_j$ to minimize the number of intersection circles with $S'$ can only decrease this number). An iterated application of this argument yields the following distance estimate in $\mathcal{S}'_N$ in terms of intersection numbers. (In the following statement, we take the convention $\log 0=0$.)

\begin{figure}
\begin{center}
\def\JPicScale{.8}
\input{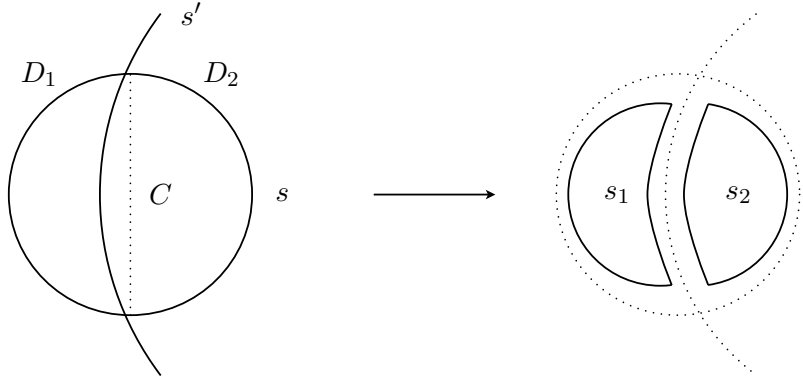}
\caption{The surgery procedure.}
\label{fig-surgery}
\end{center}
\end{figure}

\begin{lemma} \label{sphere-intersection}
There exist $K,L\in\mathbb{R}$ such that for all sphere systems $S,S'\in\mathcal{S}'_N$, we have $$d_{\mathcal{S}'_N}(S,S')\le K \log i(S,S')+L.$$
\end{lemma}

\begin{proof}
Let $s\in S$ be a sphere. Iterating the above argument yields a sequence $s=s_0,s_1,\dots,s_N$ of spheres, with $i(s_j,s_{j+1})=0$ for all $j\in\{1,\dots,N-1\}$ and $i(s_j,S')\le\frac{i(s,S')}{2^j}$. In particular, for $N:=\lceil\log_2i(S,S')\rceil$, we have $i(s_N,S')=0$. Hence the sequence $S,s,s\cup s_1,s_1,s_1\cup s_2,\dots,s_N,s_N\cup S',S'$ is a path of length $2\lceil\log_2i(S,S')\rceil+4$ joining $S$ to $S'$ in $\mathcal{S}'_N$. The lemma follows.
\end{proof}

One can finally relate the Lipschitz distance on $CV_N$ to intersection numbers.

\begin{prop} \label{Lipschitz-intersection} (Horbez \cite[Proposition 2.8]{Hor12})
There exist $K,L\in\mathbb{R}$ such that for all $\Phi,\Psi\in \text{Out}(F_N)$, we have $$\frac{1}{K}\log i(\Phi\ast_{\mathcal{S}'_N},\Psi\ast_{\mathcal{S}'_N})-L\le d_{CV_N}(\Phi b,\Psi b)\le K\log i(\Phi\ast_{\mathcal{S}'_N},\Psi\ast_{\mathcal{S}'_N})+L.$$
\end{prop}

\subsection{An Oseledets-like theorem for random products of outer automorphisms of $F_N$}\label{sec-rw-2}

When $gr(\mu)$ is elementary, we can no longer expect all elements of $F_N$ to grow exponentially fast along a typical sample path of the random walk, with the same exponential growth rate. A typical situation is the case where the support of $\mu$ only contains automorphisms that act as the identity on some proper subgroup of $F_N$: in this case, elements of $F_N$ belonging to this subgroup will not grow along any sample path of the random walk. 

However, in the case where $gr(\mu)$ is elementary, we can still provide information about the possible growth rates of elements of $F_N$ under random products of automorphisms of $F_N$. In this case, we give an analogue of a theorem due to Furstenberg and Kifer \cite{FK83} and Hennion \cite{Hen84} in the case of random products of matrices (which may be seen as a version of Oseledets' theorem). Several growth rates may arise, and we give a bound on their number. 

\subsubsection{Filtrations of $F_N$}

A \emph{filtration} of $F_N$ is a finite rooted tree $\tau$ such that

\begin{itemize}
\item associated to every node of $\tau$ is a (possibly trivial) subgroup $H\subseteq F_N$, and
\item the subgroup associated to the root of $\tau$ is $F_N$, and
\item we have $H'\subseteq H$ whenever $H'$ is a child of $H$.
\end{itemize}

A \emph{system of Lyapunov exponents} for the filtration $\tau$ is a set of real numbers $\lambda_H\ge 0$ associated to the nodes of $\tau$, such that $\lambda_{H'}\le\lambda_{H}$ whenever $H'$ is a descendant of $H$, and $\lambda_H=0$ if and only if $H$ is a leaf of $\tau$. A particular case of filtrations of $F_N$ is given by the following construction. We say that a group action on a tree is \emph{trivial} if the tree is reduced to a point. An \emph{$F_N$-chain of actions} is a finite rooted tree $\tau$ such that 

\begin{itemize}
\item associated to every node of $\tau$ is a pair $(H,T_{H})$, where $H$ is a subgroup of $F_N$ (the subgroup associated to the root of $\tau$ is $F_N$), and $T_H$ is a minimal, very small $H$-tree with dense orbits (the group $H$ might be equal to $\{e\}$, and the tree $T_H$ might be reduced to a point), and
\item all nodes whose associated action is trivial are leaves of $\tau$, and 

\item for all nodes whose associated action $(H,T_H)$ is nontrivial, the collection of subgroups $H'\subseteq F_N$ associated to the children of $(H,T_H)$ is a set of representatives of the conjugacy classes of point stabilizers in $T_H$ (in particular, the group $\{e\}$ is one of the children of $H$ as soon as the action on $T_H$ is nontrivial). 
\end{itemize}

In particular, leaves of $\tau$ are in one-to-one correspondence with trivial actions. We might have preferred not to add the trivial group to the collection of descendants of a nontrivial action, which would have led to some leaves of $\tau$ corresponding to free actions. However, it will turn out that including the trivial group in this collection is more natural for our purpose, because $e$ always has zero growth along any sample path of a random walk. Associated to any $F_N$-chain of actions is a filtration of $F_N$, obtained by forgetting the actions. The following theorem, whose proof we postpone to Section \ref{sec-count}, gives a bound on the size of any $F_N$-chain of actions.

\begin{theo}\label{count}
Any $F_N$-chain of actions has at most $N-1$ non-leaf nodes.
\end{theo}

\subsubsection{Oseledets filtrations of probability measures on $\text{Out}(F_N)$} \label{sec-Oseledets}

\begin{de}
Let $\mu$ be a probability measure on $\text{Out}(F_N)$. Let $F$ be a finitely generated subgroup of $F_N$. A filtration $\tau$ of $F$ is an \emph{Oseledets filtration} for $\mu$ if there exists a system of Lyapunov exponents $\{\lambda^{\mu}_H\}_{H\in V(\tau)}$ for $\tau$, such that for $\mathbb{P}$-almost every sample path of the random walk on $(\text{Out}(F_N),\mu)$, all nodes $H$ of $\tau$, and all elements $g\in H$ that are not conjugate into any child of $H$, we have $$\lim_{n\to +\infty}\frac{1}{n}\log||\Phi_n(g)||=\lambda^{\mu}_H.$$ For all $g\in F$, the growth rate $$\lim_{n\to +\infty}\frac{1}{n}\log||\Phi_n(g)||$$ is called the \emph{Lyapunov exponent} of $g$ for the measure $\mu$, and denoted by $\lambda^{\mu}(g)$.
\end{de}

\begin{theo}\label{Oseledets}
Let $\mu$ be a probability measure on $\text{Out}(F_N)$, having finite first moment with respect to $d_{CV_N}^{sym}$. Then there exists an Oseledets filtration for $\mu$, associated to an $F_N$-chain of actions. Moreover, for all nodes $H$ of the filtration, the conjugacy class of $H$ has finite $gr(\mu)$-orbit. 
\end{theo}

As a consequence of Theorems \ref{count} and \ref{Oseledets}, we deduce that for all probability measures $\mu$ on $\text{Out}(F_N)$ having finite first moment with respect to $d_{CV_N}^{sym}$, there exists a finite collection of (deterministic) exponents $\lambda_1,\dots,\lambda_p>0$ such that for $\mathbb{P}$-almost every sample path of the random walk on $(\text{Out}(F_N),\mu)$, and all $g\in F_N\smallsetminus\{e\}$, the limit $$\lim_{n\to +\infty}\frac{1}{n}\log||\Phi_n(g)||$$ exists and belongs to $\{0\}\cup\{\lambda_1,\dots,\lambda_p\}$. Theorem \ref{count} implies that $p\le N-1$, we will improve this bound in Section \ref{sec-count}, see Corollary \ref{growth}.

\subsubsection{Existence of Oseledets filtrations}

\paragraph*{First return measures.}

Let $\mu$ be a probability measure on $\text{Out}(F_N)$, and let $A$ be a finite index subgroup of $gr(\mu)$. The subgroup $A$ is positively recurrent for the random walk on $(\text{Out}(F_N),\mu)$. The \emph{first return measure} on $A$, denoted by $\mu^A$, is the probability measure defined as the distribution of the point where the random walk issued from the identity of $\text{Out}(F_N)$ returns for the first time to $A$. Given a sample path $(\Phi_n(\omega))_{n\in\mathbb{N}}$ of the random walk on $(\text{Out}(F_N),\mu)$, and $m\in\mathbb{N}$, we let $\tau^A_m(\omega)$ be the $(m+1)^{st}$ time $n\in\mathbb{N}$ at which we have $\Phi_n(\omega)\in A$. Notice in particular that $\tau^A_0(\omega)=0$, and $\tau^A_1(\omega)$ is the first (positive) time at which the sample path returns to the recurrent subgroup $A$. We let $$C_A:=\lim_{n\to +\infty}\frac{\tau^A_n(\omega)}{n},$$ which almost surely exists, is independent of $\omega$, and $C_A>0$ by positive recurrence of the random walk on the finite set $gr(\mu)/A$. The following proposition is a variation in our context of a classical fact about first return measures, see for example \cite[Lemma 2.3]{Kai91} or \cite[Lemma 6.10]{BQ13} where it appears in other contexts.

\begin{prop}\label{return-measure}
Let $\mu$ be a probability measure on $\text{Out}(F_N)$ which has finite first moment with respect to $d^{sym}_{CV_N}$. Let $A$ be a finite index subgroup of $\text{gr}(\mu)$ which fixes the conjugacy class of a finitely generated malnormal subgroup $H\subseteq F_N$ of rank $k$. Then $\mu^A$ has finite first moment with respect to $d^{sym}_{CV_k}$. 
\end{prop}

\begin{proof}
Since $H$ is malnormal, all elements of $A$ induce a well-defined element of $\text{Out}(H)$. We choose a basepoint $\ast_{CV_N}$ in $CV_N$, and let its $H$-minimal subtree be a basepoint $\ast_{CV_k}$ for $CV_k$. Then there exists $C>0$ such that for all $\Phi\in A$, we have $$d^{sym}_{CV_k}(\Phi\ast_{CV_k},\ast_{CV_k})\le C d^{sym}_{CV_N}(\Phi\ast_{CV_N},\ast_{CV_N}).$$ Indeed, as $\Phi$ fixes the conjugacy class of $H$, the $H$-minimal subtrees of $\ast_{CV_N}$ and $\Phi\ast_{CV_N}$ have the same quotient volumes, and the translation length of any $g\in H$ is stretched by the same amount from $\Phi\ast_{CV_N}$ to $\ast_{CV_N}$ and from $\Phi\ast_{CV_k}$ to $\ast_{CV_k}$. Denoting by $L$ the (finite) first moment of $\mu$ with respect to $d^{sym}_{CV_N}$, we have  
\begin{displaymath}
\begin{array}{rl}
\int_{A} d_{CV_k}^{sym}(\Phi\ast_{CV_k},\ast_{CV_k})d\mu^{A}(\Phi)&\le C \int_{A} d_{CV_N}^{sym}(\Phi\ast_{CV_N},\ast_{CV_N})d\mu^{A}(\Phi)\\
&=C\int_{\Omega}d_{CV_N}^{sym}(\Phi_{\tau^A_1(\omega)}(\omega)\ast_{CV_N},\ast_{CV_N})d\mathbb{P}(\omega)\\
&\le C\int_{\Omega}\sum_{i=1}^{\tau^A_1(\omega)}d_{CV_N}^{sym}(\phi_i(\omega)\ast_{CV_N},\ast_{CV_N})d\mathbb{P}(\omega)\\
&=C\sum_{i=1}^{+\infty} \int_{\{\tau_1^A(\omega)\ge i\}}d_{CV_N}^{sym}(\phi_i(\omega)\ast_{CV_N},\ast_{CV_N})d\mathbb{P}(\omega)\\
&=CL\sum_{i=1}^{+\infty}\mathbb{P}(\tau^A_1(\omega)\ge i),
\end{array}
\end{displaymath}

\noindent where the last equality follows from independence of $\{\tau^A_1\ge i\}$ and the increments $\phi_j$'s for $j\ge i$. We thus get $$\int_{A} d_{CV_k}^{sym}(\Phi\ast_{CV_k},\ast_{CV_k})d\mu^{A}(\Phi)\le CL\sum_{i=1}^{+\infty}i\mathbb{P}(\tau^A_1(\omega)=i),$$ which is finite by positive recurrence of the random walk on the finite set $gr(\mu)/A$. 
\end{proof}

\begin{prop}\label{return}
Let $\mu$ be a probability measure on $\text{Out}(F_N)$, with finite first moment with respect to $d_{CV_N}^{sym}$. Let $H$ be a finitely generated malnormal subgroup of $F_N$ of rank $k$, whose conjugacy class $[H]$ has finite $\text{gr}(\mu)$-orbit. Let $A:=\text{Stab}([H])$. 

Assume that for all probability measures $\mu'$ on $A$ with finite first moment with respect to $d_{CV_{k}}^{sym}$, there exists an Oseledets filtration of $H$ for $\mu'$. Then any Oseledets filtration of $H$ for the measure $\mu^A$ is an Oseledets filtration of $H$ for the measure $\mu$, and for all $g\in H$, we have $\lambda^{\mu}(g)=\frac{1}{C_A}\lambda^{\mu^A}(g)$.
\end{prop}

\begin{proof}
Let $\{[H]=[H_1],\dots,[H_p]\}$ be the $\text{gr}(\mu)$-orbit of the conjugacy class of $H$, and for all $i\in\{1,\dots,p\}$, let $A_i:=\text{Stab}([H_i])$. We start by showing the existence, for all $i\in\{1,\dots,p\}$, of an Oseledets filtration of $H_i$ for the measure $\mu^{A_i}$. Let $i\in\{1,\dots,p\}$. We choose an automorphism $\alpha_i\in gr(\mu)$ such that $\alpha_i([H])=[H_i]$ (with $\alpha_1=\text{id}$). Let $\mu_i$ be the measure on $A$ defined as $\mu_i:=(\text{ad}_{\alpha_i})_{\ast} \mu^{A_i}$ (where $\text{ad}_{\alpha_i}$ denotes the conjugation by $\alpha_i$). Using Proposition \ref{return-measure}, we get that $\mu_i$ has finite first moment with respect to $d_{CV_{k}}^{sym}$, so by hypothesis, there exists an Oseledets filtration of $H$ for the measure $\mu_i$. For almost every sample path $(\Psi_n^i)_{n\in\mathbb{N}}$ of the random walk on $(H,\mu_i)$, and all $g\in H$, we have $$\lim_{n\to +\infty}\frac{1}{n}\log ||\Psi_n^i(g)||=\lambda^{\mu_i}(g),$$ so for all $g'\in H_i$, we have $$\lim_{n\to +\infty}\frac{1}{n}\log ||\alpha_i\Psi_n^i\alpha_i^{-1}(g')||=\lambda^{\mu_i}(\alpha_i^{-1}(g')).$$ By definition of the measure $\mu_i$, this implies the existence of an Oseledets filtration of $H_i$ for the measure $\mu^{A_i}$, which is the $\alpha_i$-image of the Oseledets filtration of $H$ for $\mu_i$. The Lyapunov exponents $\lambda^{A_i}$ of the measure $\mu^{A_i}$ satisfy $\lambda^{{A_i}}(g')=\lambda^{\mu_i}(\alpha_i^{-1}(g'))$ for all $g'\in H_i$.
\\
\\
\indent Let now $(\Phi_n)_{n\in\mathbb{N}}$ be a sample path of the random walk on $(\text{Out}(F_N),\mu)$, and let $g\in H$. For all $n\in\mathbb{N}$, we set $g_n:=\Phi_n(g)$. For all $i\in\{1,\dots,p\}$, we let $I_i\subseteq\mathbb{N}$ be the set of all integers $n$ such that $\Phi_n([H])=[H_i]$, and we let $\tau_i(n)$ be the $n^{th}$ integer in $I_i$. The limit $$C_i:=\lim_{n\to +\infty}\frac{\tau_i(n)}{n}$$ almost surely exists, and $C_i>0$, by positive recurrence of the finite Markov chain on $\{[H_1],\dots,[H_p]\}$ induced by the random walk on $(\text{Out}(F_N),\mu)$. For all $n\in\mathbb{N}$, we have $\Phi_n (g)=\Psi_n^i (g_{\tau_i(1)})$, where $\Psi_n^i:=\phi_n\dots\phi_{\tau_i(1)+1}$. The sequence $(\Psi_n^i)_{n\in I_i}$ is a sample path of the random walk on $(A_i,\mu^{A_i})$, and therefore we have $$\lim_{\substack{n\to +\infty\\n\in I_i}}\frac{1}{n}\log ||\Phi_n (g)||=\frac{1}{C_i}\lambda^{{A_i}}(g_{\tau_i(1)}).$$ We will now prove that the limit $$\lim_{n\to +\infty}\frac{1}{n}\log ||\Phi_n(g)||$$ almost surely exists, i.e. that $\frac{1}{C_i}\lambda^{A_i}(g_{\tau_i(1)})$ does not depend on $i$. In particular, by choosing $i=1$, this will imply that $$\lim_{n\to +\infty}\frac{1}{n}\log ||\Phi_n(g)||=\frac{1}{C_A}\lambda^A(g),$$ and in particular any Oseledets filtration for $\mu^A$ is an Oseledets filtration for $\mu$.

In order to prove the above claim, we first notice that for all $\epsilon>0$, there exists $n_0\in\mathbb{N}$ such that for all $i\in\{1,\dots,p\}$, all $n\in I_i\cap [n_0,+\infty)$, and all $g\in H$, we have $$\left|\frac{1}{n}\log ||\Phi_n(g)||-\frac{1}{C_i}\lambda^{A_i}(g)\right|\le{\epsilon}.$$ Assume towards a contradiction that $\frac{1}{C_i}\lambda^{A_i}(g)\neq\frac{1}{C_{i'}}\lambda^{A_{i'}}(g)$ for some $i,i'\in\{1,\dots,p\}$. Then there exists an infinite set of integers $X$ with positive density such that for all $n\in X$, the integers $n$ and $n+1$ belong to two different sets $I_i,I_{i'}$ of the partition, with $\frac{1}{C_i}\lambda^{A_i}(g)\neq\frac{1}{C_{i'}}\lambda^{A_{i'}}(g)$. Consequently, there exists $\delta>0$ and an infinite set of integers $X'$ with positive density $\alpha>0$, such that for all $n\in X'$, we have $$\left|\frac{1}{n+1}\log ||\Phi_{n+1}(g)||-\frac{1}{n}\log||\Phi_n(g)||\right|\ge\delta,$$ and therefore $$\frac{||\Phi_{n+1}(g)||}{||\Phi_n(g)||}\ge e^{n\delta}.$$ Let $k\in\mathbb{N}$. There exists $n_k\in\mathbb{N}$ such that $e^{n_k\delta}\ge k$. For all $n\in X'\cap [n_k,+\infty)$, we have $$d_{CV_N}(\Phi_n^{-1}b,\Phi_{n+1}^{-1}b)\ge k,$$ or in other words $d_{CV_N}(\phi_{n+1}b,b)\ge k$. This implies that for all $k\in\mathbb{N}$, we have $$\mu(\{\phi\in\text{Out}(F_N)|d_{CV_N}(\phi b,b)\ge k\})\ge\alpha$$ (where we recall that $\alpha>0$ is the density of $X'$), which is impossible. So for all $i,i'\in\{1,\dots,p\}$, we have $\frac{1}{C_i}\lambda^{A_i}(g)=\frac{1}{C_{i'}}\lambda^{A_{i'}}(g)$, as claimed.
\end{proof}

\paragraph*{Proof of Theorem \ref{Oseledets}.}

We argue by induction on the rank $N$ of the free group. The claim holds true for $N=1$, so we assume that $N\ge 2$. We will first show that for almost every sample path $\mathbf{\Phi}$ of the random walk on $(\text{Out}(F_N),\mu)$, there exists an (\emph{a priori} random) filtration $\tau(\mathbf{\Phi})$ of $F_N$, together with an (\emph{a priori} random) system of Lyapunov exponents $\{\lambda^{\mathbf{\Phi}}(H)\}_{H\in V(\tau(\mathbf{\Phi}))}$, such that for all nodes $H$ of the filtration, and all $g\in H$ that are not conjugate into any child of $H$, we have $$\lim_{n\to +\infty}\frac{1}{n}\log ||\Phi_n(g)||=\lambda^{\mathbf{\Phi}}(H).$$ The fact the Lyapunov exponents are deterministic, and that the filtration can be chosen not to depend on the sample path, will be shown in the last paragraph of the proof. We keep the notations introduced in the proof of Theorems \ref{ergo-1} and \ref{ergo-2} in Section \ref{cocycle}. Recall that we have shown that $$\lim_{n\to +\infty}{\frac{1}{n}}\log ||\Phi_n(g)||=l$$ for all $g\in F_N$ which are hyperbolic in $T$, where $l$ is the drift of the random walk for the Lipschitz metric on $CV_N$. We are left understanding possible growth rates of elements of $F_N$ that are elliptic in $T$. If $l=0$, then all elements $g\in F_N$ grow subexponentially along the random walk, and we can choose $T$ to be trivial. Otherwise, the horofunction $\xi$ provided by Theorem \ref{KL} is unbounded from below, so Theorem \ref{Busemann-dense} implies that for almost every $\omega$, the tree $T(\omega)$ has dense orbits. Propositions \ref{Karlsson-Ledrappier} and \ref{nonelementary} show that we may have chosen $T$ so that all conjugacy classes of point stabilizers in $T$ have finite $\text{gr}(\mu)$-orbit. 

Let $\mathcal{C}$ be the collection of conjugacy classes of point stabilizers of $T$. All subgroups in $\mathcal{C}$ are malnormal, and they have rank at most $N-1$ by \cite[Theorem III.2]{GL95} (see Proposition \ref{index} below). Therefore, our induction hypothesis implies that for all $H\in\mathcal{C}$ of rank $k$, and all measures $\mu'$ on $\text{Out}(H)$ with finite first moment with respect to $CV_k$, there exists an Oseledets filtration of $H$ for the measure $\mu'$, which is associated to an $H$-chain of actions. Proposition \ref{return} then shows the existence of an Oseledets filtration of $H$ for the measure $\mu$, which is equal to the Oseledets filtration for $\mu^A$, where $A:=\text{Stab}([H])$. The conjugacy class of every node $H'\subseteq H$ of the filtration has finite $gr(\mu^A)$-orbit, and hence it has finite $gr(\mu)$-orbit. 

To get the desired filtration $\tau(\mathbf{\Phi})$ of $F_N$, notice that all elements of $F_N$ that do not belong to any subgroup in $\mathcal{C}$ have a Lyapunov exponent, which is greater than or equal to the Lyapunov exponent of any other element of $F_N$ (Theorem \ref{ergo-1}). We then let $\tau(\mathbf{\Phi})$ be the filtration of $F_N$ associated to the $F_N$-chain of actions whose root is the action $(F_N,T(\omega))$, to which we attach the $H$-chains of actions associated to the elliptic subgroups $H$ of $T(\omega)$ which were provided by the induction hypothesis.

We now show that the filtration $\tau(\mathbf{\Phi})$ is actually a (deterministic) Oseledets filtration for the measure $\mu$ (i.e. it is adapted to almost every sample path of the random walk on $(\text{Out}(F_N),\mu)$). It is enough to show that for all $g\in F_N$, the growth rate $\lambda^{\mathbf{\Phi}(\omega)}(g)$ of $g$ along the sample paths of the random walk on $(\text{Out}(F_N),\mu)$ is $\mathbb{P}$-essentially constant. Let $g\in F_N$. If $\lambda^{\mathbf{\Phi}(\omega)}(g)$ is not $\mathbb{P}$-essentially constant, then in particular $\mathbb{P}(\lambda^{\mathbf{\Phi(\omega)}}(g)<l)>0$ (where we recall that $l$ is the drift of the random walk on $(\text{Out}(F_N),\mu)$). Hence $g$ belongs to some subgroup $H\subseteq F_N$, whose conjugacy class has finite $\text{gr}(\mu)$-orbit. Let $A:=\text{Stab}([H])$. The induction hypothesis implies that the growth rate of $g$ along the sample paths of the random walk on $(A,\mu^A)$ is essentially constant, equal to $\lambda^{\mu^A}(g)$. Proposition \ref{return} therefore implies that $\lambda^{\mathbf{\Phi}}(g)$ is $\mathbb{P}$-essentially constant, equal to $\frac{1}{C_A}\lambda^{\mu^A}(g)$. This proves the claim.
\qed

\subsubsection{Bounding the size of $F_N$-chains of actions} \label{sec-count}

We will now prove Theorem \ref{count}, which bounds the size of any $F_N$-chain of actions. For all $T\in \overline{CV_N}$, and all $x\in T$, we define the \emph{index} $i(x):=2\text{rk}(\text{Stab}(x))+v_1(x)-2$, where $v_1(x)$ denotes the number of $\text{Stab}(x)$-orbits of directions with trivial stabilizer at $x$. This only depends on the $F_N$-orbit of $x$ in $T$. The \emph{index} $i(T)$ is then defined to be the sum of the indices of $x$ over all $F_N$-orbits of points $x\in T$. We will appeal to the following result of Gaboriau and Levitt.

\begin{prop} (Gaboriau--Levitt \cite[Theorem III.2]{GL95})\label{index}
For all trees $T\in \overline{CV_N}$, we have $i(T)\le 2N-2$. In particular, if $T$ has trivial arc stabilizers, then for all $x\in T$, we have $\text{rk}(\text{Stab}(x))\le N-1$. 
\end{prop}

\begin{proof}[Proof of Theorem \ref{count}]
We argue by induction on $N$. Every $\mathbb{Z}$-action on a tree with dense orbits is trivial, so we can assume that $N\ge 2$. Let $\tau$ be an $F_N$-chain of actions, and let $T$ be the action corresponding to the root of $\tau$. We denote by $p(\tau)$ the number of non-leaf nodes in $\tau$. Let $V$ be the collection of nodes of depth $1$ in $\tau$, which correspond to a set of representatives of the conjugacy classes of point stabilizers in $T$. For all $v\in V$, let $G_v$ be the associated subgroup of $F_N$, and let $\tau_v$ be the corresponding $G_v$-chain of actions. As $T$ has dense orbits, it follows from Proposition \ref{index} that for all $v\in V$, we have $\text{rk}(G_v)<N$. The induction hypothesis implies that for all $v\in V$, we have $p(\tau_v)\le \text{rk}(G_v)-1$, which implies that
\begin{displaymath}
\begin{array}{rl}
p(\tau)&\le 1+\sum_{v\in V}(\text{rk}(G_v)-1)\\
&<1+\frac{1}{2}\sum_{v\in V}(2\text{rk}(G_v)-1).
\end{array}
\end{displaymath} 

\noindent As arc stabilizers in $T$ are trivial, Proposition \ref{index} implies that $p(\tau)\le N-1$.
\end{proof}

\begin{ex}\label{sharp}
We construct an example of an $F_N$-chain of actions with $N-1$ nontrivial nodes, thus showing the optimality of the bound provided by Theorem \ref{count} in general. Let $T$ be a geometric $F_N$-tree with dense orbits, whose skeleton (see \cite[Definition 4.8]{Gui04} or \cite[Section 1]{Gui08}) consists of 

\begin{itemize}
\item one vertex corresponding to a minimal action with dense orbits of a subgroup of $F_N$ of rank $2$, dual to a measured lamination on a torus with a single boundary component, and
\item one vertex corresponding to a trivial action of a subgroup of $F_N$ of rank $N-1$, and
\item an edge of length $0$ joining them, whose stabilizer is cyclic, represented by the boundary curve of the torus. 
\end{itemize}

This defines a nontrivial, minimal, very small $F_N$-tree, in which a subgroup of $F_N$ of rank $N-1$ is elliptic. Repeating this construction, we get a sequence of subgroups $F_N=H_N\supseteq\dots\supseteq H_1=\mathbb{Z}$, in which the subgroup $H_i$ has rank $i$, together with minimal, very small $H_i$-trees with dense orbits, which are nontrivial as soon as $i\ge 2$, and such that for all $i\in\{2,\dots,N\}$, the subgroup $H_{i-1}$ is elliptic in $T_i$. This defines an $F_N$-chain of actions with $N-1$ non-leaf nodes, in which each node $(H_i,T_{H_i})$ with $i\ge 2$ has two children, namely the action $(H_{i-1},T_{i-1})$, and the trivial action of the trivial group. 
\end{ex}

\subsubsection{Good $F_N$-chains of actions}

Example \ref{sharp} shows that the bound on the size of an $F_N$-chain of actions provided by Theorem \ref{count} is optimal in general. We will now define a special class of \emph{good} $F_N$-chains of actions for which this bound can actually be improved. We will show that all Oseledets filtrations constructed in the proof of Theorem \ref{Oseledets} are good, which will lead to a better bound on the number of possible growth rates of elements of $F_N$ under random products of automorphisms. 

We refer to \cite{LP97} for a definition of geometric trees, see also \cite[Section 1.7]{Gui08}. Any geometric tree with dense orbits has a decomposition into a graph of actions where each nondegenerate vertex action is indecomposable \cite[Proposition 1.25]{Gui08} (the reader is referred to \cite[Section 1]{Gui08} for background material). We say that a tree in $\overline{CV_N}$ is \emph{of surface type} if it is geometric, and all its indecomposable subtrees are dual to laminations on surfaces. Let $\tau$ be an $F_N$-chain of actions. An element $g\in F_N$ is a \emph{special curve} for $\tau$ if there exists a node $(H,T_H)$ of $\tau$ corresponding to an action of surface type, such that $g$ is conjugate to an element that represents a boundary curve  of a surface dual to one of the indecomposable subtrees of $T_H$. An $F_N$-chain of actions $\tau$ is \emph{good} if all special curves of $\tau$ are elliptic in all nodes $(H,T_H)$ such that $c\in H$ (up to conjugacy). In other words, an $F_N$-chain of actions $\tau$ is good if and only if all special curves of $\tau$ are conjugate into some leaf of $\tau$. 

\begin{theo}\label{count-2}
 Any good $F_N$-chain of actions has at most $\frac{3N-2}{4}$ non-leaf nodes.
\end{theo}

\begin{theo}\label{Oseledets-2}
Let $\mu$ be a probability measure on $\text{Out}(F_N)$, having finite first moment with respect to $d_{CV_N}^{sym}$. Then there exists an Oseledets filtration for $\mu$, which is associated to a good $F_N$-chain of actions.
\end{theo}

The proof of Theorem \ref{count-2} is given in Section \ref{sec-count-2}, and the proof of Theorem \ref{Oseledets-2} is given in Section \ref{sec-Oseledets-2}. As a consequence of Theorems \ref{count-2} and \ref{Oseledets-2}, we get the following result.

\begin{cor}\label{growth}
Let $\mu$ be a probability measure on $\text{Out}(F_N)$, having finite first moment with respect to $d_{CV_N}^{sym}$. Then there exist (deterministic) $\lambda_1,\dots,\lambda_p>0$ such that for $\mathbb{P}$-almost every sample path of the random walk on $(\text{Out}(F_N),\mu)$, and all $g\in F_N\smallsetminus\{e\}$, the limit $$\lim_{n\to +\infty}\frac{1}{n}\log ||\Phi_n(g)||$$ exists and belongs to $\{0\}\cup\{\lambda_1,\dots,\lambda_p\}$. In addition, we have $p\le\frac{3N-2}{4}$.
\end{cor}

\begin{ex}\label{sharp-good}
We give an example, due to Levitt \cite{Lev09}, of a good chain of actions with $\frac{3N-2}{4}$ non-leaf nodes, thus showing that the bound in Theorem \ref{count-2} is optimal. Let $S$ be the compact, oriented surface of rank $N$ displayed on Figure \ref{fig-surface}, decomposed into $\frac{3N-2}{4}$ subsurfaces $S_i$ that are either tori with one boundary component, or spheres with $4$ boundary components. Let $H_0:=F_N$, and for all $i\in\{1,\dots,\frac{3N-2}{4}\}$, let $H_i$ be the fundamental group of the subsurface $\Sigma_i$ of $S$ obtained by removing $S_1,\dots,S_i$ from $S$ (we let $H_{\frac{3N-2}{4}}$ be the cyclic group generated by the rightmost boundary curve of the surface $S$ displayed on Figure \ref{fig-surface}). Let $T_i$ be a nontrivial $H_i$-tree with dense orbits, dual to a measured lamination on $\Sigma_i$ that is supported on $S_{i+1}$ (in particular $T_{\frac{3N-2}{4}}$ is trivial). Then the $F_N$-chain of actions displayed on Figure \ref{fig-surface} is good, because the boundary curves of $\Sigma_i$ are elliptic in all the descendants of $H_i$ that contain them. In addition, this $F_N$-chain of actions contains $\frac{3N-2}{4}$ nontrivial nodes.

The same example also shows that the bound in Corollary \ref{growth} is sharp, by letting $\mu$ be a Dirac measure supported on a diffeomorphism of $S$ that restricts to a pseudo-Anosov diffeomorphism of each surface $S_i$, with $\frac{3N-2}{4}$ different growth rates.
\end{ex}

\begin{figure}
\begin{center}
\def\JPicScale{.8}
\input{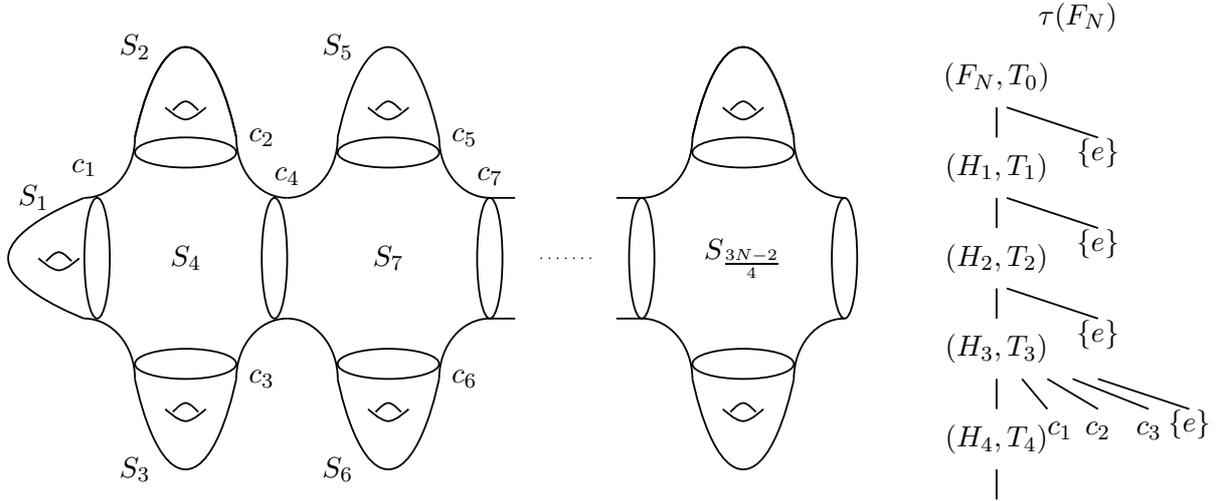}
\caption{The surface in Example \ref{sharp-good}.}
\label{fig-surface}
\end{center}
\end{figure}

\subsubsection{Existence of good Oseledets filtrations}\label{sec-Oseledets-2}

The goal of this section is to prove Theorem \ref{Oseledets-2}, by showing that the $F_N$-chain of actions constructed in the proof of Theorem \ref{Oseledets} is good.

\paragraph*{More on stationary measures on $CV_N(\infty)$.}

Let $\mathcal{Y}$ denote the collection of (finitely generated) subgroups of $F_N$ that are maximally elliptic in some simplicial tree in $\overline{CV_N}$. Given $T\in CV_N(\infty)$, we denote by $\text{Dyn}(T)$ the collection of all conjugacy classes of minimal subgroups in $\mathcal{Y}$ which act nontrivially with dense orbits on their minimal subtree in $T$. This definition makes sense by the descending chain condition satisfied by groups in $\mathcal{Y}$ \cite[Proposition 4.1]{HM09}. Subgroups whose conjugacy classes belong to $\text{Dyn}(T)$ are called \emph{dynamical subgroups} of $T$. It follows from our description of $\overline{CV_N}^{prim}$ in \cite{Hor14-1} that all lifts to $\overline{CV_N}$ of a tree in $CV_N(\infty)$ have the same dynamical subgroups. We let 
\begin{displaymath}
\Theta(T):=\left\{
\begin{array}{ll}
\text{Dyn}(T) &\text{~if~} \text{Dyn}(T) \text{~is finite}\\
\emptyset &\text{~otherwise}
\end{array}\right..
\end{displaymath} 

\noindent Measurability of $\Theta$ comes from upper semicontinuity of the quotient volume (see \cite[Section 3.3]{AK12}, where it is proved that the quotient volume of a tree $T\in\overline{cv_N}$ is equal to $0$ if and only if $T$ has dense orbits), and continuity of translation lengths. Applying Lemma \ref{disjoint-translations} to the map $\Theta$ yields the following fact.

\begin{prop}\label{stationary-measure-2}
Let $\mu$ be a probability measure on $\text{Out}(F_N)$. Then every $\mu$-stationary measure on $CV_N(\infty)$ is concentrated on the set of trees which either have infinitely many dynamical subgroups, or all of whose dynamical subgroups have finite $\text{gr}(\mu)$-orbits.
\qed
\end{prop} 

We now determine $\text{Dyn}(T)$ in the case where $T\in CV_N(\infty)$ is of surface type.

\begin{lemma}\label{dyn}
Let $T\in CV_N(\infty)$ be a tree of surface type with dense orbits. Then $\text{Dyn}(T)$ is equal to the set of stabilizers of the indecomposable subtrees of $T$.
\end{lemma}

\begin{proof}
The tree $T$ admits a transverse covering $\mathcal{Z}$ by indecomposable subtrees (see \cite[Section 1]{Gui08} for definitions), whose skeleton has cyclic (or trivial) edge groups, and each tree in $\mathcal{Z}$ is dual to a minimal lamination on a surface \cite[Proposition 1.25]{Gui08}. Let $H\subseteq F_N$ be the stabilizer of one of these indecomposable subtrees $T_H\in\mathcal{Z}$. Let $F\in\text{Dyn}(T)$. The $F$-minimal subtree $T_F$ of $T$ inherits a transverse covering, given by the intersections of $T_F$ with the subtrees in $\mathcal{Z}$. As $F$ acts with dense orbits on $T_F$ by assumption, the intersection $T_F\cap T_H$ has dense $F\cap H$-orbits. As $T_H$ is indecomposable, this implies by \cite[Theorem 4.4]{Rey11-2} that either $T_F\cap T_H$ contains at most one point, or else that $T_F\cap T_H=T_H$, and $F\cap H$ has finite index in $H$. By minimality of $F$, we have $T_F\cap T_H=T_H$ for exactly one of the subtrees $T_H$ in the family $\mathcal{Z}$. As groups in $\mathcal{Y}$ do not have proper finite index extensions, this implies that $F\cap H=H$, and $F$ is the stabilizer of one of the indecomposable subtrees of $T$. Therefore, the set $\text{Dyn}(T)$ consists of the conjugacy classes of these stabilizers. 
\end{proof}

As a consequence of Proposition \ref{stationary-measure-2} and Lemma \ref{dyn}, we get the following fact.

\begin{cor}\label{good}
Let $\mu$ be a probability measure on $\text{Out}(F_N)$. Then every $\mu$-stationary measure on $CV_N(\infty)$ is concentrated on the set of trees $T\in CV_N(\infty)$ such that either 

\begin{itemize}
\item the tree $T$ is not of surface type, or
\item the tree $T$ is of surface type, and all conjugacy classes of the stabilizers of its indecomposable components have finite $\text{gr}(\mu)$-orbits.\qed
\end{itemize} 
\end{cor}

\paragraph*{Proof of Theorem \ref{Oseledets-2}.}

We prove that the $F_N$-chain of actions $\tau$ constructed in the proof of Theorem \ref{Oseledets} is good. We keep the notations from this proof. Arguing by induction again, we can assume that all filtrations $\tau_H$ are associated to good $H$-chains of actions. Proposition \ref{Karlsson-Ledrappier}, together with Corollary \ref{good}, shows that if $T$ is of surface type, then we might assume that the conjugacy classes of all stabilizers of the indecomposable components of $T$ (which are dual to minimal foliations on surfaces) have finite $gr(\mu)$-orbit. Let $c\in F_N$ represent a boundary curve of a surface dual to one of the indecomposable components of $T$. Then $c$ is the intersection of a point stabilizer of $T$ with a dynamical subgroup of $T$, which implies that the $gr(\mu)$-orbit of the conjugacy class of $c$ is finite, and therefore $c$ grows subexponentially along the random walk. In particular, the element $c$ belongs to one of the leaves of $\tau$, thus showing that $\tau$ is good.
\qed

\subsubsection{Bounding the size of good $F_N$-chains of actions}\label{sec-count-2}

The aim of this section is to prove Theorem \ref{count-2}, which provides a bound on the size of good $F_N$-chains of actions. Our proof is inspired from Levitt's similar statement in \cite{Lev09} for counting growth rates of a single automorphism of $F_N$.

\paragraph*{An Euler characteristic formula for small graph of groups decompositions.}

\begin{lemma}\label{euler}
Let $\mathcal{G}$ be a graph of groups decomposition of $F_N$, whose edge groups are (at most) cyclic. Denote by $V$ the number of vertices of $\mathcal{G}$, by $E_0$ the number of edges with trivial stabilizer, and by $R$ the sum of the ranks of the vertex stabilizers of $\mathcal{G}$. Then $N=R+E_0-V+1$.
\end{lemma}

Our proof of Lemma \ref{euler} relies on the following classical result.

\begin{lemma} \label{unfold}(Shenitzer \cite{She55}, Swarup \cite{Swa86}, Stallings \cite{Sta91}, Bestvina--Feighn \cite[Lemma 4.1] {BF94})
Let $\mathcal{G}$ be a graph of groups decomposition of $F_N$, whose edge groups are (at most) cyclic. Then there exists an edge $e$ in $\mathcal{G}$ with nontrivial stabilizer $G_e$, adjacent to a vertex $v$, and a free splitting of $G_v$ of the form $G_v=G_e\ast A$, so that if $e'\neq e$ is another edge adjacent to $v$ in $\mathcal{G}$, then $G_{e'}$ is conjugate into $A$.
\end{lemma}

Lemma \ref{unfold} shows that we can "unfold" the edge $e$ and get another graph of groups decomposition of $F_N$ having fewer edges with nontrivial stabilizer, in which the vertex $v$ is replaced by a vertex with stabilizer equal to $A$, which has corank $1$ in $G_v$.

\begin{proof}[Proof of Lemma \ref{euler}]
Using Lemma \ref{unfold}, and arguing by downward induction on the number of edges with nontrivial stabilizer, we reduce to the case where all edges in $\mathcal{G}$ have trivial stabilizer (each unfolding operation decreases $R$ by $1$, and increases $E_0$ by $1$). By iteratively collapsing all edges in a maximal subtree of $\mathcal{G}$ (such a collapse decreases both $E_0$ and $V$ by $1$), we reduce to the case where the underlying graph of $\mathcal{G}$ is a rose, in which case Lemma \ref{euler} clearly holds.
\end{proof}

\paragraph*{Proof of Theorem \ref{count-2}.} 

Let $\tau$ be a good $F_N$-chain of actions. Let $k(\tau)$ be the rank of the subgroup of the abelianization of $F_N$ generated by the leaf groups of $\tau$. We will show by induction on $N$ that the number $p(\tau)$ of non-leaf nodes of $\tau$ satisfies $$p(\tau)\le\frac{3N-k(\tau)-2}{4}.$$ We let $T$ be the $F_N$-tree associated to the root of $\tau$. We denote by $\{G_v\}$ the collection of all subgroups associated to the children of the root in $\tau$, whose conjugates are the point stabilizers in $T$ by definition.
\\
\\
\noindent \textit{Case 1} : The tree $T$ is not of surface type.\\

\indent We refer the reader to \cite{GL10} for an introduction to (relative) JSJ decompositions (and Grushko decompositions in particular). By \cite[Theorem 3.11]{Hor14-1}, there exists a two-edge free splitting $S$ of $F_N$ in which all $G_v$'s are elliptic. This implies that any Grushko decomposition of $F_N$ relative to the collection $\{G_v\}$ has at least two edges with trivial stabilizer. We denote by $\mathcal{C}$ the collection of all conjugacy classes of subgroups of $F_N$ that are elliptic in all Grushko decompositions of $F_N$ relative to the collection $\{G_v\}$. Notice that for all $H\in\mathcal{C}$, the point stabilizers of the action of $H$ on its minimal subtree $T_H$ are conjugates of the $G_v$'s. We let $\tau_H$ be the $H$-chain of actions whose root corresponds to either

\begin{itemize}
\item the action $(H,T_H)$, where $T_H$ denotes the $H$-minimal subtree in $T$, if this action is nontrivial, or 
\item the action associated to $G_v$ in $\tau$ if $H=G_v$ for some $v\in V$,
\end{itemize}

\noindent to which we attach the trees $\tau_v$ corresponding to the subgroups $G_v$ conjugate into $H$. We have $$p(\tau)\le 1+\sum_{H\in\mathcal{C}}p(\tau_H).$$ As all subgroups $H\in\mathcal{C}$ have rank at most $N-1$, our induction hypothesis shows that $$p(\tau_H)\le\frac{3\text{rk}(H)-k(\tau_H)-2}{4}.$$ Let $G$ be a Grushko decomposition of $F_N$ relative to $\mathcal{C}$, and denote by $V$ (resp. $E$) the number of vertices (resp. of edges) in the graph of groups $G$. By collapsing edges to points if necessary, we can assume that no vertex of $G$ has trivial stabilizer. Therefore, we have 
\begin{displaymath}
\begin{array}{rl}
\sum_{H\in\mathcal{C}}(3\text{rk}(H)-2)&= 3\sum_{H\in\mathcal{C}}\text{rk}(H)-2V\\
&=3(N-E+V-1)-2V\\
&=3N-3E+V-3\\
&\le 3N-2E-2\\
&\le 3N-6
\end{array}
\end{displaymath}

\noindent because $V\le E+1$ by connectedness of $G$, and $E\ge 2$. 

In addition, any relation between elements in the subgroup generated by the leaves of $\tau_H$ still holds true when viewing these elements as elements of $F_N$, so $$k(\tau)\le\sum_{H\in\mathcal{C}}k(\tau_H).$$ Combining the above inequalities, we get that $$1+\sum_{H\in\mathcal{C}}p(\tau_H)\le\frac{3N-k(\tau)-2}{4},$$ and we are done in this case.
\\
\\
\textit{Case 2} : The tree $T$ is geometric, and only contains minimal surface components.\\

Then $T$ is dual to a graph of actions $\mathcal{G}$ having a vertex associated to each orbit of indecomposable subtrees $Y$ of $T$, a vertex associated to each conjugacy class of elliptic subgroup $H$ of $T$, and an edge joining $Y$ to $H$ whenever the fixed point of $H$ belongs to $Y$. All edges in $\mathcal{G}$ have cyclic (possibly trivial) stabilizer. Notice that $\mathcal{G}$ might not be minimal, in the case where some point stabilizer in $T$ is cyclic (and corresponds to a boundary curve of one of the surfaces dual to an indecomposable subtree of $T$) and extremal. We denote by $\mathcal{C}$ the set of conjugacy classes of point stabilizers in $T$, and by $\mathcal{C}_1$ (resp. $\mathcal{C}_{\ge 2}$) the set of conjugacy classes in $\mathcal{C}$ which have valence $1$ (resp. valence at least $2$) in $\mathcal{G}$. It follows from Proposition \ref{index} that $$\sum_{H\in\mathcal{C}_1}(2\text{rk}(H)-1)+\sum_{H\in\mathcal{C}_{\ge 2}}(2\text{rk}(H))\le 2N-2,$$ from which we deduce that $$\sum_{H\in\mathcal{C}_1}(3\text{rk}(H)-2)+\frac{1}{2}|\mathcal{C}_1|+\sum_{H\in\mathcal{C}_{\ge 2}}(3\text{rk}(H)-2)+2|\mathcal{C}_{\ge 2}|\le 3N-3,$$ or in other words $$\sum_{H\in\mathcal{C}}(3\text{rk}(H)-2)\le 3N-3-(\frac{1}{2}|\mathcal{C}_1|+2|\mathcal{C}_{\ge 2}|).$$ If $$\sum_{H\in\mathcal{C}}(3\text{rk}(H)-2)\le 3N-6,$$ then we are done as in Case 1. Otherwise, we either have $|\mathcal{C}_{\ge 2}|=1$ and $|\mathcal{C}_{1}|=0$, or $|\mathcal{C}_{\ge 2}|=0$ and $|\mathcal{C}_1|\le 4$. 
\\
\\
\indent We now assume that $|\mathcal{C}_{\ge 2}|=1$ and $|\mathcal{C}_1|=0$. In this case, the graph of actions $\mathcal{G}$ consists of a central vertex corresponding to an elliptic subgroup $H\in\mathcal{C}_{\ge 2}$, which is attached to $k$ indecomposable subtrees (dual to laminations on surfaces) by edges with trivial or cyclic stabilizers. We denote by $\sigma_1,\dots,\sigma_k$ the ranks of the stabilizers of these minimal components, and by $E_0$ the number of edges with trivial stabilizer in $\mathcal{G}$.  For all $i\in\{1,\dots,k\}$, we have $\sigma_i\ge 2$, and Lemma \ref{euler} implies that $$N=\sum_{i=1}^k(\sigma_i-1)+\text{rk}(H)+E_0.$$ Again, we get that $3\text{rk}(H)-2\le 3N-6$, except possibly if $k=1$ and $\sigma_1=2$ (and $E_0=0$). In this case, the corresponding surface is a torus having a single boundary component (there are no minimal foliations on spheres having at most $3$ boundary components, nor on projective planes with at most $2$ boundary components, nor on a Klein bottle with one boundary component \cite{CV91}). This contradicts the fact that $H\in\mathcal{C}_{\ge 2}$.
\\
\\
\indent We now assume that $|\mathcal{C}_{\ge 2}|=0$, and $|\mathcal{C}_1|\le 4$. In this case, the graph of actions $\mathcal{G}$ is a tree that consists of a single vertex $v_0$ corresponding to a connected surface $S$, attached to vertices corresponding to subgroups in $\mathcal{C}$ by edges with trivial or cyclic stabilizer. Denoting by $m$ the number of boundary components of $S$ (which is also equal to the number of edges with nontrivial stabilizer in $\mathcal{G}$), and by $s$ the rank of the fundamental group of $S$, we get from Lemma \ref{euler} that $$N=\sum_{H\in\mathcal{C}}\text{rk}(H)+s-m.$$ This implies that $$\sum_{H\in\mathcal{C}}(3\text{rk}(H)-2)\le 3N+m-3s,$$ which is bounded by $3N-6$ as soon as $3s-m\ge 6$ (in which case we conclude as in Case $1$). 

If $S$ is a nonorientable surface of genus $g\ge 1$, then $s=g+m-1$, and the condition $3s-m\ge 6$ is equivalent to $3g+2m\ge 9$. This condition is satisfied, except in the cases where either $g=1$ and $m\le 2$, or $g=2$ and $m=1$. However, as we have already mentioned, there is no minimal measured lamination on a projective plane having at most $2$ boundary components, nor on a Klein bottle with one boundary component.

If $S$ is an orientable surface of genus $g$, then $s=2g+m-1$, and the condition $3s-m\ge 6$ is equivalent to $6g+2m\ge 9$. This condition is satisfied, except in the cases where either $g=1$ and $m=1$, or $g=0$ and $m\le 4$.   

If $g=m=1$, then $S$ is a torus with a single boundary component, whose fundamental group $F_2$ is amalgamated in the corresponding splitting of $F_N$ to a group $G_v$ along its boundary curve $c$. The curve $c$ is trivial in the abelianization of $F_N$ (it is a commutator), while it is not in the abelianization of $G_v$ (it represents a primitive element in $G_v$). As $\tau$ is good, the element $c$ belongs to a leaf of the subtree $\tau_v$ of $\tau$, whose root subgroup is $G_v$. Hence $k(\tau)<k(\tau_v)$, and as $3\text{rk}(\tau_v)-2=3N-5$, we deduce that $$3\text{rk}(G_v)-k(\tau_v)-2\le 3N-k(\tau)-6,$$ which is enough to conclude.

If $g=0$ and $m\le 4$, then $S$ is a sphere with $4$ boundary components (there is no minimal lamination on a sphere having at most $3$ boundary components). Using goodness of $\tau$, and the fact that the product of the elements corresponding to its boundary curves is equal to $1$, we get that $$k(\tau)<\sum_{H\in\mathcal{C}}k(\tau_H)$$ (where $\tau_H$ denotes the subtree of $\tau$ whose root subgroup is $H$), and we conclude similarly.
\qed 

\subsubsection{Random products of mapping classes of surfaces}

In the case where $gr(\mu)$ is contained in the mapping class group $\text{Mod}(S)$ of a compact, orientable, hyperbolic surface $S$ with nonempty totally geodesic boundary, the length of the isotopy class of any simple closed curve on $S$, measured in any hyperbolic metric on $S$, is bi-Lipschitz equivalent to the length of the corresponding element of the (free) fundamental group of $S$. Given an oriented compact surface $S$ with genus $g$ and $s$ boundary components, the \emph{complexity} of $S$ is defined as $\xi(S):=3g+s-3$. In the case where $s\ge 1$, the rank $N$ of the fundamental group of $S$ satisfies $\xi(S)=\frac{3N-2}{4}$. Corollary \ref{growth} therefore yields the following statement, which refines Karlsson's results in \cite{Kar12}. 

\begin{cor}\label{mcg}
Let $S$ be a compact hyperbolic oriented surface with nonempty totally geodesic boundary. Let $\mu$ be a probability measure on $\text{Mod}(S)$, with finite first moment with respect to Thurston's asymmetric metric on the Teichmüller space of $S$. Then there exist (deterministic) $\lambda_1,\dots,\lambda_p> 0$ such that for almost every sample path $(\Phi_n)_{n\in\mathbb{N}}$ of the random walk on $(\text{Mod}(S),\mu)$, all simple closed curves $\alpha$ on $S$, and all hyperbolic metrics $\rho$ on $S$, either $l_{\rho}(\Phi_n(\alpha))$ grows subexponentially, or there exists $i\in\{1,\dots,p\}$ such that $$\lim_{n\to +\infty}\frac{1}{n}\log {l_{\rho}(\Phi_n(\alpha))}=\lambda_i.$$ In addition, we have $p\le \xi(S)$.
\end{cor}

By combining our arguments with Karlsson's \cite{Kar12}, Corollary \ref{mcg} can also be proved in the case of a closed orientable surface. Proper subsurfaces play the role of proper free factors of $F_N$, and the filtration of $F_N$ provided by Theorem \ref{growth} is replaced by a decomposition of the surface into subsurfaces.

\begin{theo}\label{theo-mcg}
Let $S$ be a compact hyperbolic oriented surface with (possibly empty) totally geodesic boundary. Let $\mu$ be a probability measure on $\text{Mod}(S)$ having finite first moment with respect to Thurston's asymmetric metric on the Teichmüller space of $S$. Then there exists a decomposition of $S$ into subsurfaces $\{S_i\}_{1\le i\le k}$, and for all $i\in\{1,\dots,k\}$, a \emph{Lyapunov exponent} $\lambda_i\ge 0$, so that for almost every sample path $(\Phi_n)_{n\in\mathbb{N}}$ of the random walk on $(\text{Mod}(S),\mu)$, all simple closed curves $\alpha$ on $S$, and all hyperbolic metrics $\rho$ on $S$, the limit $$\lim_{n\to +\infty}\frac{1}{n}\log l_{\rho}(\Phi_n(\alpha))$$ exists, and is equal to the maximum of the Lyapunov exponents of a subsurface $S_i$ crossed by $\alpha$ (in the case where $\alpha$ is one of the curves defining the decomposition of $S$, the limit is equal to $0$). The number of positive Lyapunov exponents is bounded by $\xi(S)$.
\end{theo}

We call such a decomposition an \emph{Oseledets decomposition} of $S$ for the measure $\mu$.

\begin{proof}[Sketch of proof of Theorem \ref{theo-mcg}]
The horoboundary of the Teichmüller space $Teich(S)$ of $S$, equipped with Thurston's asymmetric metric, has been identified by Walsh with the space $\mathcal{PMF}$ of projectified measured foliations. Applying Lemma \ref{disjoint-translations} to the map $\Theta$ that sends a measured foliation to its support (which is a disjoint union of subsurfaces of $S$), we see that all $\mu$-stationary measures on $\mathcal{PMF}$ are concentrated on the set of measured foliations whose supports have finite $gr(\mu)$-orbit. Following Karlsson's argument in \cite{Kar12}, we get for almost every sample path $(\Phi_n)_{n\in\mathbb{N}}$ of the random walk on $(\text{Mod}(S),\mu)$ the existence of $\eta\in\mathcal{PMF}$ such that for all simple closed curves $\alpha$ on $S$ such that $i(\eta,\alpha)>0$, and all hyperbolic metrics $\rho$ on $S$, we have $$\lim_{n\to +\infty}\frac{1}{n} \log l_{\rho}(\Phi_n(\alpha))=l,$$ where $l$ denotes the drift of the random walk on $(\text{Mod}(S),\mu)$. In addition, we can assume that the support of $\eta$ has finite $gr(\mu)$-orbit. The condition $i(\eta,\alpha)=0$ is equivalent to $\alpha$ lying in the complement $S'$ of the support of $\eta$ in $S$ (or $\alpha$ being one of the boundary curves of this support). Arguing by induction on the complexity of the surface, we get the existence of a decomposition of $S'$, which is an Oseledets decomposition for the first return measure on the stabilizer of $S'$. Arguing as in Proposition \ref{return}, we get that the decomposition of $S$ obtained by adding the boundary curves of $S'$ to this decomposition of $S'$ is an Oseledets decomposition for $\mu$.
\end{proof}

\bibliographystyle{amsplain}
\bibliography{/Users/Camille/Documents/Bibliographie}

\end{document}